\documentclass[twocolumn]{autart}    
\usepackage[hidelinks]{hyperref}
\usepackage{amsfonts}
\usepackage{indentfirst}
\usepackage{amssymb}
\usepackage{graphicx}
\usepackage[authoryear,longnamesfirst]{natbib}
\usepackage{changes}
\usepackage{color}
\usepackage{epsfig}
\usepackage{mathrsfs}
\usepackage{epstopdf}
\usepackage{amsmath}
\usepackage[noend]{algpseudocode}
\usepackage{algorithmicx,algorithm}
\usepackage{enumitem}
\usepackage[title]{appendix}
\usepackage[none]{hyphenat}
\usepackage{stfloats}
\DeclareMathOperator*{\diag}{diag}
\DeclareMathOperator*{\esssup}{ess\,sup}

\newtheorem{proposition}{Proposition}
\newtheorem{assumption}{Assumption}
\newtheorem{remark}{Remark}
\newtheorem{definition}{Definition}
\newtheorem{lemma}{Lemma}
\newtheorem{theorem}{Theorem}
\newtheorem{corollary}{Corollary}
\newtheorem{example}{Example}

\newenvironment{proof}{\emph{Proof:}}

\begin{document}
\setlength{\parindent}{2em}
\begin{frontmatter}
\title{Small Gain Theorem-Based Robustness Analysis of Discrete-Time MJLSs with the Markov Chain on a Borel Space and Its Application to NCSs} 

\author[sdnu]{Chunjie Xiao}\ead{$xiaocj6\_sd@163.com$},
\author[sdnu]{Ting Hou \corauthref{cor}}
\corauth[cor]{Corresponding author.}\ead{$ht\_math@sina.com$},
\author[sdust]{Weihai Zhang}\ead{$w\_hzhang@163.com$},
\author[scut]{Feiqi Deng}\ead{$aufqdeng@scut.edu.cn$},
\address[sdnu]{School of Mathematics and Statistics, Shandong Normal University,
Jinan 250014, Shandong Province, China}
\address[sdust]{College of Electrical Engineering and Automation, Shandong University of Science and Technology,
Qingdao 266590, Shandong Province, China}
\address[scut]{School of Automation Science and Engineering, South China University of Technology,
Guangzhou 510640, Guangdong Province, China}


\begin{abstract}             
This paper is concerned with the robustness of discrete-time Markov jump linear systems (MJLSs) with the Markov chain on a Borel space. For this general class of MJLSs, a small gain theorem is first established and subsequently applied to derive a lower bound of the stability radius.
On this basis, with the aid of the extended bounded real lemma and Schur complements, the robust stability problems for the MJLSs are tackled via linear matrix inequality (LMI) techniques.
The novel contribution, primarily founded on the scenario where the state space of the Markov chain is restricted in a continuous set, lies in the formulation of a griding approach.
The approach converts the existence problem of solutions of an inequality related to $H_{\infty}$ analysis, which is an infinite-dimensional challenge, into a finite-dimensional LMI feasibility problem.
As an application, within the framework of MJLSs, a robustness issue of the sampled-data systems is addressed by using a Markov chain, which is determined by the initial distribution and the stochastic kernel, to model transmission delays existing in networked control systems (NCSs). Finally, the feasibility of the results is verified through two examples.
\sloppy{}
\end{abstract}
\end{frontmatter}

\section{Introduction}
Markov jump linear systems (MJLSs) have been extensively studied since they were introduced in the early 1960s (see, \cite{Krasovskii1961}) because of their capacity to model dynamic systems whose structures experience random abrupt changes.
An MJLS is composed of a set of subsystems, enabling the random variations of the internal dynamics of the modeled actual system to be achieved through transitions among these subsystems driven by a Markov chain.
Existing results surrounding MJLSs can be found in monographs \cite{BookCosta2005,BookDragan2010,BookDragan2014}. Moreover, numerous applications have been reported in the specialized literature across various fields, including economics (see, \cite{Blair1975IJC}), engineering (see, \cite{Bueno2023SMC}), power generation (see, \cite{Vargas2013,Lin2016}), and networks (see, \cite{ZhangLiqian2005, Impicciatore2024TAC}).

Robust stability refers to the ability of a dynamic system to maintain stability in the presence of uncertainties and external disturbances.
A substantial amount of work has been conducted regarding uncertain MJLSs to address robustness matters (see, \cite{Boukas2001TAC, Souza2006TAC, Aberkane2015SIAM, Lun2019Auto}). Among these studies, the parameter uncertainties include those in the matrices of the system state-space model (see, \cite{Aberkane2015SIAM}) or in the transition probability matrix of the
Markov chain (see, \cite{Lun2019Auto}), generally assumed to be of the norm-bounded type (see, \cite{Boukas2001TAC}) or the polytopic type (see, \cite{Souza2006TAC}).
In particular, for MJLSs subject to norm-bounded uncertainties, the classical quantitative index of robust stability known as stability radius has been extended.
Most findings concerning the estimation of lower bounds for the stability radius are primarily based on specific versions of the small gain theorem (see, \cite{BookDragan2010, BookDragan2014, Aberkane2015SIAM}), which consequently provide sufficient conditions for robust stability.
Additionally, necessary conditions for robust stability have been established for discrete-time systems (see, \cite{todorov2012new}) and for continuous-time systems (see, \cite{Todorov2013Auto}), respectively.

The work mentioned above predominantly focuses on MJLSs with the Markov chain taking values in a finite set, commonly referred to as finite MJLSs.
In this paper, we are interested in discrete-time MJLSs with the Markov chain on a Borel space. The motivations for exploring this class of systems are twofold. On one side, this general setting allows the state space of the Markov chain to
extend beyond a finite set, and it can be an uncountably infinite set, such as a continuous set. On the other side, this extension offers a more accurate model for capturing complex real-world phenomena. For instance, as we will see in Example \ref{exampledelay2}, a continuous-valued Markov chain would be a more suitable choice for modeling the transmission delays in  sampled-data systems. In addition, the solar thermal receiver model studied in \cite{Costa2015} can also be employed to support this assertion.

Over the last decade, the control problems of MJLSs with the Markov chain on a Borel space have gained attention (see, \cite{Ioannis2014, Costa2014,Costa2015, Costa2016, Costa2017,Masashi2018,Xiao2023}), led by the exploration of stability issues discussed in \cite{Li2012}.
Here, we do not aim to be exhaustive but would like to highlight a recent study reported in \cite{Xiao2023}, known as the finite horizon bounded real lemma (BRL), which associates the estimation of the $H_{\infty}$ norm of the input-to-output operator with the existence of the stabilizing solution of the coupled algebraic Riccati equations (coupled-AREs). In this paper, the lemma is further developed by introducing a problem related to the solvability of a coupled nonlinear inequality involving matrix-valued functions (see Lemma \ref{lemmaBRL}).
Within the finite MJLS framework, the linear matrix inequality (LMI)-based BRL (for example, see \cite{Seiler2003}), which is recognized as an effective and practical tool for robust analysis (\cite{Morais2018Auto}), can be obtained by using the Schur complements to the coupled nonlinear inequalities.
With regard to the more general framework where the state space of the Markov chain is a Borel space, an extension of the classical Schur complements is achieved in this paper.
Nonetheless, there are significant challenges in solving the coupled  inequality that involves matrix-valued functions, especially when the Markov chain takes values in an uncountably infinite set.

This paper tackles the robust stability problem of discrete-time MJLSs with the Markov chain on a Borel space. The obtained findings extend the existing research on finite MJLSs;
however, these results are not mere trivial extensions. Ensuring the measurability, boundedness, and integrability of the solutions to the involved coupled equations (such as the coupled-AREs) or inequalities presents more challenges than in the finite case. The main contributions can be summarized as follows:

\begin{itemize}
\item Based on the established BRL,  the small gain theorem has been extended to systems with the Markov chain on a general state space through a Lyapunov-type argument, with its formulation being connected to $H_{\infty}$ analysis. During the derivation, the invertibility of certain operators has been confirmed by applying the infinite-dimensional operator theory.

\item As a direct application of the small gain theorem, a lower bound of the stability radius of systems with norm-bounded uncertainties has been provided in Theorem \ref{lowboundsta963}, expressed in terms of the inverse of the $H_{\infty}$ norm of an input-output operator.

\item Based on Theorem \ref{lowboundsta963}, a sufficient condition for determining the robust stability of the system has been established by estimating the $H_{\infty}$ norm of the corresponding input-output operator through the usage of the BRL. Specifically, the stability margin specific to finite MJLS has been determined by solving a convex optimization problem constrained by a finite number of LMIs. Regarding MJLSs with continuous-valued Markov jump parameters, a gridding technique has been employed to convert the solvability problem of an inequality involving matrix-valued functions, which is associated with the $H_{\infty}$ norm constraint of an input-output operator, into a feasibility problem with a finite number of LMIs.
    In this way, the infinite-dimensional problem has been transformed into a finite-dimensional one, thereby establishing an LMI-based BRL.
    In addition, the feasibility and conservatism of this gridding technique have been explored.

\item  A robustness problem for sampled-data systems with uncertain perturbations existing in controller parameters has been reformulated and addressed within the framework of MJLSs.
\end{itemize}

This paper is structured as follows: In Section \ref{SecPreliminaties}, we provide the essential notations, model description, and several auxiliary results.
The key contribution of Section \ref{SecSmallGain} lies in establishing the small gain theorem, which is then applied in Section \ref{SecStabilityRadii} to derive a lower bound for the stability radius. Section \ref{SecRobustStability} is devoted to presenting feasible sufficient conditions for robust stability of MJLSs, organized into three subsections:
In Subsection \ref{SecFiniteMJLS}, a brief discussion of the finite MJLS case is conducted;
The case of MJLSs with continuous-valued Markov jump parameters is studied in Subsection \ref{SecRobustofConCase}, aided by a proposed griding technique; Furthermore, the feasibility and conservatism of the griding technique are analysed in Subsection \ref{subsecFeaConGriding}.
Section \ref{SecApptoNCS} shows two examples based on the NCSs. The work is concluded in Section \ref{SecConclusions}.

\section{Preliminary Results}\label{SecPreliminaties}
\subsection{Notations and an Auxiliary Result}
As usual, $\mathbf{E}\{\cdot\}$ stands for the mathematical expectation.
$\mathbb{N}$ is the set of all nonnegative integers.
$\mathbb{N}^{+}:=\mathbb{N}/\{0\}.$
For integers $k_{1}\leq k_{2}$, $\overline{k_{1},k_2}:=\{k_{1},k_{1}+1,k_{1}+2,\cdots,k_{2}\}$.
$\mathbb{R}^{n}$ denotes the $n$-dimensional real Euclidean space with the Euclidean norm $\|\cdot\|_{\mathbb{R}^{n}}$.
For Banach spaces $\mathcal{H}$ and $\bar{\mathcal{H}}$ with norms $\|\cdot\|_{\mathcal{H}}$ and $\|\cdot\|_{\bar{\mathcal{H}}}$, $\mathbf{B}(\mathcal{H},\bar{\mathcal{H}})$
represents the Banach space of all bounded linear operators from $\mathcal{H}$ to $\bar{\mathcal{H}}$, equipped with the induced norm
$\|\mathcal{L}\|_{\mathbf{B}(\mathcal{H},\bar{\mathcal{H}})}:=\sup_{h\in\mathcal{H},\ \|h\|_{\mathcal{H}}=1}\{{\|\mathcal{L}h\|_{\bar{\mathcal{H}}}}\},$ where $\mathcal{L}\in\mathbf{B}(\mathcal{H},\bar{\mathcal{H}})$.
Particularly, $\mathbb{R}^{n\times m}:=\mathbf{B}(\mathbb{R}^{m},\mathbb{R}^{n})$
is the space composed of $(n \times m)$-dimensional real matrices with the norm defined as $\|M\|_{\mathbb{R}^{n\times m}}=[\lambda_{max}(M^{T}M)]^{\frac{1}{2}}$,
where $\lambda_{max}(M)$ and $M^{T}$ indicate the largest eigenvalue and transpose of $M\in{\mathbb{R}^{n\times m}}$.
$\mathbb{S}^{n}:=\{M\in{\mathbb{R}^{n\times n}}|M^{T}=M\}$;
$\mathbb{S}^{n+*}:=\{M\in{\mathbb{S}^{n}}|M\text{ is a positive definite matrix}\}$.
$I$ represents the identity matrix with appropriate dimension.
$\diag\{M_{1},  M_{2}, \cdots, M_{n}\}$ denotes a block diagonal matrix  whose diagonal term equals $M_{i}$ for any index $i\in\overline{1, n}.$

Throughout this paper, $\Theta$ is assumed to be a Borel subset of a separable and complete metric space.
The Borel space $(\Theta, \mathcal{B}(\Theta))$ is defined as $\Theta$,
together with its Borel $\sigma$-algebra $\mathcal{B}(\Theta)$.
$\mu$ is a $\sigma$-finite measure on $\mathcal{B}(\Theta)$.
One can refer to \cite{Xiao2023} for more details on Borel space.

Now we proceed with giving some matrix-valued function spaces needed in the sequel. Let $\mathcal{H}^{n\times m}$
($\mathcal{SH}^{n}$) denote the space of measurable matrix-valued functions
$P(\cdot): \Theta\rightarrow \mathbb{R}^{n\times m}$ ($P(\cdot): \Theta\rightarrow \mathbb{S}^{n}$).
$\mathcal{H}^{n\times m}_{1}=\big\{P\in\mathcal{H}^{n\times m}\big| \|P\|_{\mathcal{H}^{n\times m}_{1}}:=\int_{\Theta}\|P(\ell)\|_{\mathbb{R}^{n\times m}}\mu(d\ell)<\infty\big\}$;
$\mathcal{H}^{n\times m}_{\infty}=\big\{P\in\mathcal{H}^{n\times m}\big| \|P\|_{\mathcal{H}^{n\times m}_{\infty}}=\esssup\{\|P(\ell)\|_{\mathbb{R}^{n\times m}}$, $\ell\in{\Theta}\}<\infty\big\}.$
$(\mathcal{H}_{1}^{n\times m},\|\cdot\|_{\mathcal{H}^{n\times m}_{1}})$ and $(\mathcal{H}^{n\times m}_{\infty},\|\cdot\|_{\mathcal{H}^{n\times m}_{\infty}})$ are Banach spaces (see \cite{Costa2014}).
$\mathcal{SH}^{n}_{1}=\big\{P\in\mathcal{SH}^{n}\big| \|P\|_{\mathcal{H}^{n\times n}_{1}}<\infty\big\}$;
$\mathcal{SH}^{n}_{\infty}=\big\{P\in\mathcal{SH}^{n}\big| \|P\|_{\mathcal{H}^{n\times n}_{\infty}}<\infty\big\}$.
$(\mathcal{SH}^{n}_{\infty},\|\cdot\|_{\mathcal{H}^{n\times n}_{\infty}})$ is an ordered Banach space with the order induced by
$\mathcal{H}_{\infty}^{n+}=\big\{P \in \mathcal{SH}_{\infty}^{n} | P(\ell)\geq 0\ \mu\text{-almost everywhere on } \Theta\ (\mu\text{-}a.e.)\big\}$.
$\mathcal{H}_{\infty}^{n+*}=\big\{P \in \mathcal{SH}_{\infty}^{n} |P(\ell)\gg 0 \ \mu\text{-}a.e. \big\}$,
where $P(\ell)\gg 0$ indicates that $P(\ell) \geq \xi I$ for some $\xi>0$;
$\mathcal{H}_{\infty}^{n-*}=\big\{P \in \mathcal{SH}_{\infty}^{n} |P(\ell)\ll 0 \ \mu\text{-}a.e. \big\}$,
where $P(\ell)\ll 0$ indicates that $P(\ell) \leq -\xi I$ for some $\xi>0$.
In the paper, all properties of Borel-measurable functions, including inequalities and equations, should be understood in the sense of $\mu$-$a.e.$, or for $\mu$-almost all $\ell\in\Theta$, unless otherwise specified.

The following result can be viewed as a generalization of the classical Schur complements.

\begin{proposition}\label{Schurcom}
For arbitrary $P_{1}\in \mathcal{SH}^{n}_{\infty},\ P_{2}\in \mathcal{H}^{n\times m}_{\infty},$ and $P_{3}\in \mathcal{SH}^{m}_{\infty}$, the following are equivalent:\\
$(i)$~$P=\left[
       \begin{array}{cc}
         P_{1} & P_{2} \\
        P_{2}^{T}& P_{3} \\
       \end{array}
     \right]\in \mathcal{H}^{(n+m)+*}_{\infty};$\\
$(ii)$~$P_{3}\in \mathcal{H}^{m+*}_{\infty}$ and
$P_{1}-P_{2}P_{3}^{-1}P_{2}^{T}\in \mathcal{H}^{n+*}_{\infty};$\\
$(iii)$~$P_{1}\in \mathcal{H}^{n+*}_{\infty}$ and
$P_{3}-P_{2}^{T}P_{1}^{-1}P_{2}\in \mathcal{H}^{m+*}_{\infty}.$
\end{proposition}
\begin{proof}
See Appendix \ref{AppA}.
\end{proof}

\subsection{Model Description and Basic Definitions}

On a probability space $(\Omega, \mathfrak{F}, \mathbb{P})$,
define a Markov chain $\{\vartheta(k), k\in\mathbb{N}\}$ taking values in $\Theta$
with the initial distribution given by a probability measure $\mu_{0}$
and the stochastic kernel $\mathbb{G}(\cdot,\cdot)$ satisfying
\begin{equation*}\label{308regularcon}
\mathbb{G}(\vartheta(k),\Lambda)=\mathbb{P}(\vartheta(k+1)\in \Lambda|\vartheta(k))
\text{ almost surely }(a.s.),
\end{equation*}
$\ k\in\mathbb{N},
\ \Lambda\in\mathcal{B}(\Theta).$
The assumption regarding $(\Theta, \mathcal{B}(\Theta))$ being a Borel space can ensure that the defined Markov chain is well constructed (see \cite{Kallenberg2002}).
Additionally,  we make the following assumptions throughout this paper:
\begin{assumption}\label{assMarkov}
$(i)$~The initial distribution $\mu_{0}$ of the Markov chain is absolutely continuous w.r.t. $\mu$;\\
$(ii)$~For any $\ell\in\Theta$,
$\mathbb{G}(\ell,\cdot)$ has a density $g(\ell,\cdot)$ with respect to (w.r.t.) $\mu$,
that is,
$\mathbb{G}(\ell,\Lambda)
=\int_{\Lambda}g(\ell,t)\mu(dt)$ for any $\Lambda\in\mathcal{B}(\Theta)$.
\end{assumption}

By considering $(i)$ of Assumption \ref{assMarkov},
it follows from the Radon-Nikodym theorem (see, for example, Theorem 13.2 in \cite{BookBillingsley1995}) that there is  a density $\nu_{0}$ of $\mu_{0}$ such that  $\mu_{0}(\Lambda)=\int_{\Lambda}\nu_{0}(\ell)\mu(d\ell)$ for any $\Lambda\in{\mathcal{B}(\Theta)}$.
Define a sequence of functions $\{\nu_{k},k\in\mathbb{N}\}$ as
\begin{equation}\label{n215}
\nu_{k+1}(\ell)=\int_{\Theta}\nu_{k}(t)g(t,\ell)\mu(dt),\ \ell\in\Theta
\end{equation}
with the initial value $\nu_{0}$.
It can be seen from, for example, \cite{Costa2016} or Pages 25-26 of \cite{BookLerma2003}, that
for any $\Lambda\in{\mathcal{B}(\Theta)}$,
\begin{equation}\label{ProMa206}
\int_{\Lambda}\nu_{k}(\ell)\mu(d\ell)
=\mathbb{P}\{\vartheta(k)\in{\Lambda}\},\ k\in\mathbb{N}.
\end{equation}
That is, for each $k\in\mathbb{N}$, $\nu_{k}$ serves as the density of the distribution of $\vartheta(k)$ w.r.t. $\mu$.
Therefore, the constructed Markov chain $\{\vartheta(k), k\in\mathbb{N}\}$
is completely determined by its initial distribution and stochastic kernel.

Consider the following  MJLS:
\begin{equation} \label{414SGsystem}
\left\{
\begin{array}{ll}
x(k+1)=A(\vartheta(k))x(k)+B(\vartheta(k))v(k), \\
z(k)=C(\vartheta(k))x(k)+D(\vartheta(k))v(k),\ k\in\mathbb{N}
\end{array}
\right.
\end{equation}
with $x(0)=x_{0}$  being a deterministic vector in $\mathbb{R}^{n}$ and $\vartheta(0)=\vartheta_{0}$ being a random variable.
$x(k)\in{\mathbb{R}^{n}}$, $v(k)\in{\mathbb{R}^{r}}$, and $z(k)\in{\mathbb{R}^{r}}$ are the system state, input, and output. For $k\in\mathbb{N},$
$\mathfrak{F}_{k}$ stands for the $\sigma$-field generated by $\{\vartheta_{0}, \vartheta(1), \cdots, \vartheta(k)\}$.

In this paper, Regarding MJLS \eqref{414SGsystem}, assume that:
\begin{assumption}\label{Assumption1}
$(i)$~$A=\{A(\ell)\}_{\ell\in\Theta}\in\mathcal{H}^{n\times n}_{\infty}$,
$B=\{B(\ell)\}_{\ell\in\Theta}\in\mathcal{H}^{n\times r}_{\infty}$,
$C=\{C(\ell)\}_{\ell\in\Theta}\in\mathcal{H}^{r\times n}_{\infty}$,
$D=\{D(\ell)\}_{\ell\in\Theta}\in\mathcal{H}^{r\times r}_{\infty}$;\\
$(ii)$~The Markov chain $\{\vartheta(k),k\in\mathbb{N}\}$ is directly accessible;\\
$(iii)$~$C(\ell)^{T}D(\ell)=0\ \mu\text{-}a.e..$
\end{assumption}

Review the concepts of stability and detectability proposed in  \cite{Xiao2023} and \cite{Xiao2024}, as they will be utilized in this paper.

\begin{definition}
Consider the autonomous MJLS termed $(A|\mathbb{G})$:
\begin{equation*}
x(k+1)=A(\vartheta(k))x(k),\ k\in\mathbb{N}.
\end{equation*}
$(A|\mathbb{G})$  is said to be
exponentially mean-square stable (EMSS)
if there exist $\beta >0$, $\alpha \in(0,1)$ such that
 for any initial conditions $(x_0,\vartheta_{0})$, we have
$$\mathbf{E}\{\left.\|x(k)\|_{\mathbb{R}^{n}}^{2}\right\} \leq \beta \alpha^{k}\|x_{0}\|_{\mathbb{R}^{n}}^{2},\ k \in \mathbb{N}.$$ Further, MJLS \eqref{414SGsystem} is said to be internally stable if $(A|\mathbb{G})$ is EMSS.
\end{definition}

\begin{definition}\label{defexde1009}
Consider the following MJLS termed $(A;C|\mathbb{G})$:
$$\left\{
\begin{array}{ll}
x(k+1)=A(\vartheta(k))x(k),\ \\
z(k)=C(\vartheta(k))x(k),\ k\in\mathbb{N}.
\end{array}
\right.$$
$(A;C|\mathbb{G})$ is said to be detectable if there exists $H\in\mathcal{H}^{n\times r}_{\infty}$ such that $(A
+HC|\mathbb{G})$ is EMSS.
\end{definition}

\subsection{Some Preliminary Robustness Results}

In this subsection, we present some preliminary results regarding robustness of MJLSs with the Markov chain on Borel space  $(\Theta,\mathcal{B}(\Theta))$.

We begin by defining some Hilbert spaces:
Define $\tilde{l}^{2}(\mathbb{N};\mathbb{R}^{r})$ as the space of all measurable functions $v: \mathbb{N}\times \Omega\rightarrow \mathbb{R}^{r}$
with
$$\|v\|_{\tilde{l}^{2}(\mathbb{N};\mathbb{R}^{r})}
:=[\sum_{k=0}^{\infty}\mathbf{E}(\|v(k)\|_{\mathbb{R}^{r}}^{2})]^{\frac{1}{2}}
<\infty,$$
and it is known that $\tilde{l}^{2}(\mathbb{N};\mathbb{R}^{r})$ is a Hilbert space. In what follows, by $l^{2}(\mathbb{N};\mathbb{R}^{r})$ we denote the space of all $v\in\tilde{l}^{2}(\mathbb{N};\mathbb{R}^{r})$ with that
$v$ is a sequence of $\mathbb{R}^{r}$-valued random variables and that $v(k)$ is $\mathfrak{F}_{k}$-measurable for any $k\in \mathbb{N}$.  $l^{2}(\mathbb{N};\mathbb{R}^{r})$ is a Hilbert space with norm $\|v\|_{l^{2}(\mathbb{N};\mathbb{R}^{r})}
:=[\sum_{k=0}^{\infty}\mathbf{E}(\|v(k)\|_{\mathbb{R}^{r}}^{2})]
^{\frac{1}{2}}$ since it is closed in $\tilde{l}^{2}(\mathbb{N};\mathbb{R}^{r})$.

In MJLS \eqref{414SGsystem}, the input $v=\{v(k),k\in\mathbb{N}\}$ is a random sequence which belongs to $l^{2}(\mathbb{N};\mathbb{R}^{r})$. Let $\phi_{x}(k,0,\vartheta_{0},v)$ and $\phi_{z}(k,0,\vartheta_{0},v)$ denote the state and output of MJLS \eqref{414SGsystem} with the initial conditions $x(0)=0$ and $\vartheta_{0}$. It can be proved similarly to Corollary 2 in \cite{Xiao2023} that when MJLS \eqref{414SGsystem} is internally stable, $\{\phi_{x}(k,0,\vartheta_{0},v),k\in\mathbb{N}\}
\in{l^{2}(\mathbb{N};\mathbb{R}^{n})}$
for any $v\in{l^{2}(\mathbb{N};\mathbb{R}^{r})},$ and further, $\{\phi_{z}(k,0,\vartheta_{0},v),k\in\mathbb{N}\}\in{l^{2}(\mathbb{N};\mathbb{R}^{r})}.$
Therefore, we can define a bounded linear operator $\mathfrak{L}$ by
$\mathfrak{L}(v)(k)=\phi_{z}(k,0,\vartheta_{0},v),
\ k\in\mathbb{N},$
and it is called the input-to-output operator associated with MJLS \eqref{414SGsystem}. Alternatively,  MJLS \eqref{414SGsystem} is said to be the state space realization of operator $\mathfrak{L}$.
To define the $H_{\infty}$ norm of  $\mathfrak{L}$, we introduce the following performance function:
\begin{equation}\label{711J}
J_{\gamma}(x_{0},\vartheta_{0},v)=
\|\mathfrak{L}(v)\|_{l^{2}(\mathbb{N};\mathbb{R}^{r})}^{2}-\gamma\|v\|_{l^{2}(\mathbb{N};\mathbb{R}^{r})}^{2},\ \gamma>0.
\end{equation}
When MJLS \eqref{414SGsystem} is internally stable, the $H_{\infty}$ norm of the input-to-output operator $\mathfrak{L}$ associated with the system \eqref{414SGsystem} is defined as
\begin{equation*}
\begin{split}
\|\mathfrak{L}\|_{\infty}&=\inf\big\{\sqrt{\gamma}|
J_{\gamma}(0,\vartheta_{0},v)<0 \text{ for any } \vartheta_{0} \text{ and}\\
& \text{ nonzero input } v \in l^{2}(\mathbb{N};\mathbb{R}^{r}), \gamma>0
\big\}.
\end{split}
\end{equation*}
It should be remarked that $\|\mathfrak{L}\|_{\infty}$ is
irrelevant to the distribution of $\vartheta_{0}$ and it measures the influence of the disturbances in the worst-case scenario. Moreover, the $H_{\infty}$ norm can also be defined as the form of the $l^{2}$-induced norm of $\mathfrak{L}$ (see \cite{Xiao2023}).

For brevity, in the sequel, given $\gamma>0$, for any $P\in\mathcal{SH}^{n}_{\infty}$ and $\ell\in\Theta$, define the following operators:
\begin{align}
& \mathcal{E}(P)(\ell):=\int_{\Theta}g(\ell,t)P(t)\mu(dt),\nonumber \\
& \mathcal{T}_{A}(P)(\ell):=A(\ell)^{T}\mathcal{E}(P)(\ell)A(\ell),\label{OpTA355} \\
& \Psi_{1}(P)(\ell):=\mathcal{T}_{A}(P)(\ell)+C(\ell)^{T}C(\ell),\nonumber\\
&\Psi_{2}(P)(\ell):=A(\ell)^{T}\mathcal{E}(P)(\ell)B(\ell),
\nonumber\\
&\Psi_{3}^{\gamma}(P)(\ell):=\gamma I-
\mathcal{T}_{B}(P)(\ell)
-D(\ell)^{T}D(\ell),\nonumber\\
&\mathcal{F}^{\gamma}(P)(\ell):=-\Psi_{3}^{\gamma}(P)(\ell)^{-1}
\Psi_{2}(P)(\ell)^{T},\nonumber\\
&\bar{\mathcal{F}}^{\gamma}(P)(\ell):=-[\gamma I-\mathcal{T}_{B}(P)(\ell)]^{-1}\Psi_{2}(P)(\ell)^{T},\label{371opFbar}
\end{align}
where $\mathcal{T}_{B}(\cdot)$ is given by \eqref{OpTA355} with $A$ substituted by $B.$

One can verify that $\mathcal{E}(P)\in\mathbf{B}(\mathcal{SH}_{\infty}^{n},
\mathcal{SH}_{\infty}^{n}),$ $\mathcal{T}_{A}\in\mathbf{B}(\mathcal{SH}_{\infty}^{n},
\mathcal{SH}_{\infty}^{n}),$ $\Psi_{1}\in\mathbf{B}(\mathcal{SH}_{\infty}^{n},
\mathcal{SH}_{\infty}^{n}),$ $\Psi_{2}\in\mathbf{B}(\mathcal{SH}_{\infty}^{n}, \mathcal{H}_{\infty}^{n\times r}),$ $\Psi_{3}^{\gamma}\in\mathbf{B}(\mathcal{SH}_{\infty}^{n}, \mathcal{SH}_{\infty}^{r}).$
Moreover, it is worth emphasizing that in this paper, prior to  employing the operator $\mathcal{F}^{\gamma}(P)$, we always check that either $\Psi_{3}^{\gamma}(P)(\ell)\gg 0 \ \mu\text{-}a.e.$ or $\Psi_{3}^{\gamma}(P)(\ell)\ll 0 \ \mu\text{-}a.e.$. Indeed, it can be seen from  Remark 7 in \cite{Xiao2024} that these conditions are essential to guarantee $[\Psi_{3}^{\gamma}(P)]^{-1}\in\mathcal{SH}_{\infty}^{r}$.
Further, $\mathcal{F}^{\gamma}\in\mathbf{B}(\mathcal{SH}_{\infty}^{n}, \mathcal{H}_{\infty}^{r\times n})$
is obtained.

Throughout this paper, the following assumptions are made on the Markov chain, and their existence, as shown in \cite{Xiao2023}, plays a crucial role in building the BRL, which will be applied in the sequel.
\begin{assumption}\label{Asspositive}
$(i)$~$\nu_{0}(\ell)>0$ $\mu\text{-}a.e.$;\\
$(ii)$~$\int_{\Theta}g(t,\ell)\mu(dt)>0\ \mu\text{-}a.e..$
\end{assumption}

\begin{proposition}\label{341proposition}
~$(i)$~If Assumption \ref{Asspositive} is fulfilled,
then $\nu_{k}(\ell)>0$ $\mu\text{-}a.e.$ for each $k\in\mathbb{N}$.\\
~~~~$(ii)$~Conversely, if  $\nu_{k}(\ell)>0$ $\mu\text{-}a.e.$ for each $k\in\mathbb{N}$, then  we have that $\int_{\Theta}g(t,\ell)\mu(dt)>0\ \mu\text{-}a.e..$.
\end{proposition}
\begin{proof}
See Appendix \ref{Pronondegen}.
\end{proof}
\begin{remark}
Note that $\{\vartheta(k),k\in\mathbb{N}\}$ having a positive distribution at each moment is necessary for deriving the BRL,
for example, refer to \cite{Aberkane2015SIAM} for finite MJLSs and
\cite{Xiao2023} for the Borel setting.
If Assumption \ref{Asspositive} is fulfilled,
from \eqref{ProMa206} and Proposition \ref{341proposition}, it follows that for any $\Lambda\in{\mathcal{B}(\Theta)}$ with $\mu(\Lambda)>0$, $$\mathbb{P}\{\vartheta(k)\in{\Lambda}\}>0,\ \forall k\in\mathbb{N}.$$
Thereby, we assume that Assumption \ref{Asspositive} are
prior conditions before giving the BRL.
We also mention that when $\Theta$ is specialized as a countably infinite set
$\mathbb{N}^{+}$,
the Assumption \ref{Asspositive} are retrieved; in this case, $(i)$ and $(ii)$
correspond to the initial distribution $\mathbb{P}(\vartheta_0=i)>0, i\in\mathbb{N}^{+}$ being positive and the transition probability matrix $\mathbf{P}$ being non-degenerate, respectively.
Furthermore, for any $i\in\mathbb{N}^{+}$ and  $k\in\mathbb{N},$ $\mathbb{P}(\vartheta(k)=i)>0$ directly follows from $\mathbb{P}(\vartheta(k)=i)=\mathbb{P}(\vartheta_0=i) \mathbf{P}^{k}$.
\end{remark}

We proceed with showing the BRL, which describes the internal stability and the $H_{\infty}$ norm of MJLS \eqref{414SGsystem} in terms of either the existence of the stabilizing solution of the coupled-AREs or, equivalently, the solvability of the coupled nonlinear inequalities.

\begin{lemma}\label{lemmaBRL}
Given any $\gamma>0$, the following are equivalent:\\
$(i)$~ MJLS \eqref{414SGsystem} is internally stable
and satisfies $\|\mathfrak{L}\|_{\infty}<\sqrt{\gamma}$;\\
$(ii)$~There exists a unique stabilizing solution (see Definition 4 in \cite{Xiao2023}) $\bar{P}\in\mathcal{H}_{\infty}^{n+}$ to the coupled-AREs
\begin{equation}\label{lmiARE}
\bar{P}(\ell)=\Psi_{1}(\bar{P})(\ell)
-\Psi_{2}(\bar{P})(\ell)
[-\Psi_{3}^{\gamma}(\bar{P})(\ell)]^{-1}
\Psi_{2}(\bar{P})(\ell)^{T}
\end{equation}
such that
$\Psi_{3}^{\gamma}(\bar{P})(\ell)\gg0\ \mu\text{-}a.e.;$\\
$(iii)$~There exists a solution $\hat{P}\in\mathcal{H}_{\infty}^{n+*}$ satisfying
\begin{equation}\label{LMI741}
\hat{P}(\ell)-\Psi_{1}(\hat{P})(\ell)-\Psi_{2}(\hat{P})(\ell)
\Psi_{3}^{\gamma}(\hat{P})(\ell)^{-1}\Psi_{2}(\hat{P})(\ell)^{T}\gg 0
\end{equation}
such that
$\Psi_{3}^{\gamma}(\hat{P})(\ell)\gg 0\ \mu\text{-}a.e..$\\
Moreover, if items $(i)$-$(iii)$ hold, then
$\hat{P}-\bar{P}\in\mathcal{H}_{\infty}^{n+*}.$
\end{lemma}
\begin{proof}
See Appendix \ref{AppB}.
\end{proof}

\section{The Small Gain Theorem}\label{SecSmallGain}
The main goal of this section is to establish the small gain theorem. Before that, we first give the following result, which is fundamental for the derivation of the small gain theorem.
\begin{proposition}\label{usedinSMG}
If  MJLS \eqref{414SGsystem} is internally stable and
the input-to-output operator associated with MJLS \eqref{414SGsystem} satisfies $\|\mathfrak{L}\|_{\infty}<1$, then the input-to-output operator $(\mathfrak{I}-\mathfrak{L})$ is invertible and the state space realization of $(\mathfrak{I}-\mathfrak{L})^{-1}$ is internally stable. Here, $\mathfrak{I}$ is the identity element of $\mathbf{B}(l^{2}(\mathbb{N};\mathbb{R}^{r}),
l^{2}(\mathbb{N};\mathbb{R}^{r})).$
\end{proposition}
\begin{proof}
See Appendix \ref{AppenPro2}.
\end{proof}

Consider the systems
\begin{align}
&\left\{
\begin{array}{ll}
x_{1}(k+1)=A_{1}(\vartheta(k))x_{1}(k)+B_{1}(\vartheta(k))v_{1}(k), \\
z_{1}(k)=C_{1}(\vartheta(k))x_{1}(k),\ k\in\mathbb{N},
\end{array}
\right.
\label{575systemg1}\\
&\left\{
\begin{array}{ll}
x_{2}(k+1)=A_{2}(\vartheta(k))x_{2}(k)+B_{2}(\vartheta(k))v_{2}(k), \\
z_{2}(k)=C_{2}(\vartheta(k))x_{2}(k)+D_{2}(\vartheta(k))v_{2}(k).
\end{array}
\right. \label{583systemg2}
\end{align}


Suppose that MJLSs \eqref{575systemg1} and \eqref{583systemg2} are internally stable. Let $\mathfrak{L}_{1}$ and $\mathfrak{L}_{2}$ be the input-to-output operators  from $l^{2}(\mathbb{N};\mathbb{R}^{r})$ to $l^{2}(\mathbb{N};\mathbb{R}^{r})$ associated with MJLSs \eqref{575systemg1} and \eqref{583systemg2}, respectively.
\begin{figure}
\centering
\vspace{-5pt}
\includegraphics[width=0.15\textwidth]{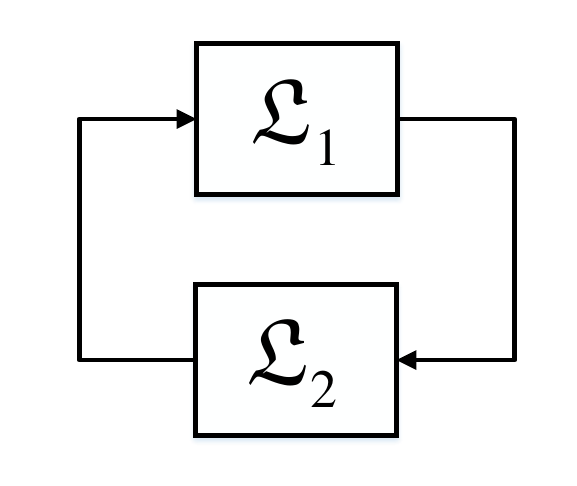}\\
\caption{The interconnection.}\label{InCo597}
\end{figure}
Next, we turn our attention to the interconnection in Fig. \ref{InCo597}, in which \eqref{575systemg1} and \eqref{583systemg2} are coupled by
$v_{2}=z_{1}$ and $v_{1}=z_{2}.$
For any $\ell\in\Theta$, define
$$
\hat{A}(\ell)=\left[
   \begin{array}{cc}
A_{1}(\ell)
+B_{1}(\ell)D_{2}(\ell)C_{1}(\ell)
 & B_{1}(\ell)C_{2}(\ell) \\
     B_{2}(\ell)C_{1}(\ell) & A_{2}(\ell) \\
   \end{array}
 \right].
 $$
The dynamics for the interconnection
is described as
\begin{equation}\label{InCo599}
\hat{x}(k+1)
=\hat{A}(\vartheta(k))\hat{x}(k),\ k\in\mathbb{N},
\end{equation}
where
$\hat{x}(k)=\left[
                  \begin{array}{cc}
                    x_{1}(k)^{T}&
                      x_{2}(k)^{T} \\
                  \end{array}
                \right]^{T}$.

We are now in a position to give the main result of this section, called the small gain theorem.
\begin{theorem}\label{smg}
If the following assumptions hold: \\
$(i)$~Systems \eqref{575systemg1} and \eqref{583systemg2} are internally stable;\\
$(ii)$~$\|\mathfrak{L}_{1}\|_{\infty} \|\mathfrak{L}_{2}\|_{\infty}<1.$\\
Then the interconnection \eqref{InCo599} is EMSS.
\end{theorem}
\begin{proof}
It can be easily seen from $(ii)$ that $\|\mathfrak{L}_{1}\mathfrak{L}_{2}\|_{\infty}<1.$
For any $\ell\in\Theta$,
define
$\hat{B}(\ell)=
\left[
  \begin{array}{cc}
    [B_{1}(\ell)D_{2}(\ell)]^{T} & B_{2}(\ell)^{T} \\
  \end{array}
\right]^{T}.$
If the state space representation of $\mathfrak{L}_{1}\mathfrak{L}_{2}$:
\begin{equation}\label{694nistat}
\left\{
\begin{array}{ll}
\begin{split}
\hat{x}(k+1)
=&\left[
   \begin{array}{cc}
  A_{1}(\vartheta(k))
  & B_{1}(\vartheta(k))C_{2}(\vartheta(k)) \\
    0 & A_{2}(\vartheta(k)) \\
   \end{array}
 \right]
 \hat{x}(k)\\
& +\hat{B}(\vartheta(k))v_{2}(k),
\end{split}\\
z(k)=C_{1}(\vartheta(k))x_{1}(k),\  k\in\mathbb{N}
\end{array}
\right.
\end{equation}
is internally stable, applying Proposition \ref{usedinSMG}, one can get that the state space realization of $(\mathfrak{I}-\mathfrak{L}_{1}\mathfrak{L}_{2})^{-1}$, given as
\begin{equation}\label{694nistatre}
\left\{
\begin{array}{ll}
\hat{x}(k+1)
=\hat{A}(\vartheta(k))\hat{x}(k)+\hat{B}(\vartheta(k))v(k), \\
z(k)=C_{1}(\vartheta(k))x_{1}(k)+v(k),\  k\in\mathbb{N},
\end{array}
\right.
\end{equation}
is internally stable.
Note that system \eqref{694nistatre} being internally stable is equivalent to system \eqref{InCo599} being EMSS.
To obtain the desired result,
we will prove that system \eqref{694nistat} is internally stable, i.e.,
system
\begin{equation}\label{check736}
\hat{x}(k+1)
=\left[
   \begin{array}{cc}
 A_{1}(\vartheta(k))
 & B_{1}(\vartheta(k))C_{2}(\vartheta(k)) \\
    0 & A_{2}(\vartheta(k)) \\
   \end{array}
 \right]
 \hat{x}(k)
\end{equation}
is EMSS.
Since systems \eqref{575systemg1} and \eqref{583systemg2} are internally stable, according to Theorem 4 in \cite{Xiao2023}, there exist $\xi_1>0$, $\xi_2>0$, $P_{1}\in\mathcal{H}^{n+*}_{\infty}$, and $P_{2}\in\mathcal{H}^{n+*}_{\infty}$ such that
\begin{equation}\label{743P1}
\mathcal{T}_{A_{1}}(P_{1})(\ell)-P_{1}(\ell)\leq -\xi_{1} I\  \mu\text{-}a.e.
\end{equation}
and
\begin{equation}\label{747P2}
\mathcal{T}_{A_{2}}(P_{2})(\ell)-P_{2}(\ell)\leq -\xi_{2} I\  \mu\text{-}a.e.,
\end{equation}
where $\mathcal{T}_{A_{1}}$ and $\mathcal{T}_{A_{2}}$ are given by \eqref{OpTA355} with $A=A_{1}$ and $A=A_{2}$, respectively.
Set $V_{1}(k,x_{1}(k),\vartheta(k))$ $=x_{1}(k)^{T}P_{1}(\vartheta(k))x_{1}(k)$ and $V_{2}(k,x_{2}(k),\vartheta(k))=x_{2}(k)^{T}P_{2}(\vartheta(k))x_{2}(k).$
In view of $$\mathbf{E}\left\{P_{1}(\vartheta(k+1))|\mathfrak{F}_{k}\right\}
=\mathcal{E}(P_{1})(\vartheta(k))\ a.s.,$$
one can establish that
\begin{equation*}
\begin{split}
&\mathbf{E}\left\{V_{1}(k+1,x_{1}(k+1),\vartheta(k+1))
-V_{1}(k,x_{1}(k),\vartheta(k))\right\}\\
&=\mathbf{E}\{ x_{1}(k)^{T}[\mathcal{T}_{A_{1}}(P_{1})(\vartheta(k))-P_{1}(\vartheta(k))]x_{1}(k) \\& +2x_{1}(k)^{T}A_{1}(\vartheta(k))^{T}\mathcal{E}(P_{1})(\vartheta(k))
B_{1}(\vartheta(k))C_{2}(\vartheta(k))x_{2}(k)\\
&+x_{2}(k)^{T}C_{2}(\vartheta(k))^{T}B_{1}(\vartheta(k))^{T}
\mathcal{E}(P_{1})(\vartheta(k))B_{1}(\vartheta(k))\\
&\ \ \cdot C_{2}(\vartheta(k))x_{2}(k)
\}.
\end{split}
\end{equation*}
Furthermore, for any $\zeta_{0}>0$, it holds that
\begin{equation*}
\begin{split}
&\mathbf{E}\{2x_{1}(k)^{T}A_{1}(\vartheta(k))^{T}\mathcal{E}(P_{1})(\vartheta(k))
B_{1}(\vartheta(k))C_{2}(\vartheta(k))x_{2}(k)\}\\
&\leq\mathbf{E}\{\zeta_{0}^{-1}\|A_{1}(\vartheta(k))^{T}\mathcal{E}(P_{1})(\vartheta(k))
B_{1}(\vartheta(k))C_{2}(\vartheta(k))\|^{2}_{\mathbb{R}^{n\times n}}\\
&\ \ \ \cdot \|x_{2}(k)\|^{2}_{\mathbb{R}^{n}}
+\zeta_{0}\|x_{1}(k)\|^{2}_{\mathbb{R}^{n}}
\}.
\end{split}
\end{equation*}
Additionally, from the inequality
$\|\mathcal{E}(P_{1})(\ell)\|_{\mathbb{R}^{n\times n}}
\leq \|P_{1}\|_{\mathcal{H}_{\infty}^{n \times n}}\ \mu\text{-}a.e.,$
it follows that
$$\|A_{1}(\ell)^{T}\mathcal{E}(P_{1})(\ell)
B_{1}(\ell)C_{2}(\ell)\|^{2}_{\mathbb{R}^{n\times n}}
\leq \zeta_{1}^{2}\ \mu\text{-}a.e.$$
and
$$\|C_{2}(\ell)^{T}B_{1}(\ell)^{T}
\mathcal{E}(P_{1})(\ell)B_{1}(\ell)C_{2}(\ell)\|_{\mathbb{R}^{n\times n}}\leq \zeta_{2} \ \mu\text{-}a.e.,$$
where $\zeta_{1}=\|A_{1}\|_{\mathcal{H}^{n\times n}_{\infty}}\|B_{1}\|_{\mathcal{H}^{n\times r}_{\infty}}\|C_{2}\|_{\mathcal{H}^{r\times n}_{\infty}}
\|P_{1}\|_{\mathcal{H}^{n\times n}_{\infty}}$ and
$\zeta_{2}=\|B_{1}\|_{\mathcal{H}^{n\times r}_{\infty}}^{2}\|C_{2}\|_{\mathcal{H}^{r\times n}_{\infty}}^{2}
\|P_{1}\|_{\mathcal{H}^{n\times n}_{\infty}}.$
Through the obtained relations, we arrive at
\begin{equation}\label{V1cha734}
\begin{split}
&\mathbf{E}\left\{V_{1}(k+1,x_{1}(k+1),\vartheta(k+1))
-V_{1}(k,x_{1}(k),\vartheta(k))\right\}\\
&\leq (\zeta_{2}+\zeta_{0}^{-1}\zeta_{1}^{2})\mathbf{E}\{\|x_{2}(k)\|^{2}_{\mathbb{R}^{n}}\}
-(\xi_{1}-\zeta_{0})\mathbf{E}\{\|x_{1}(k)\|^{2}_{\mathbb{R}^{n}}\}.
\end{split}
\end{equation}
Similarly, from \eqref{747P2} one can get that
\begin{align}
&\mathbf{E}\left\{V_{2}(k+1,x_{2}(k+1),\vartheta(k+1))
-V_{2}(k,x_{2}(k),\vartheta(k))\right\} \nonumber \\
&\leq -\xi_{2}\mathbf{E}\{\|x_{2}(k)\|^{2}_{\mathbb{R}^{n}}\}.
\label{V2cha744}
\end{align}
Next, take $0<\zeta_{0}<\xi_{1}$ and let $V(k,x_{1}(k),x_{2}(k),\vartheta(k))
=V_{1}(k,x_{1}(k),\vartheta(k))+\frac{2
(\zeta_{2}+\zeta_{0}^{-1}\zeta_{1}^{2})}{\xi_2}
V_{2}(k,x_{2}(k),\vartheta(k)).$
Combining \eqref{V1cha734} with \eqref{V2cha744} gives
\begin{align*}
&\mathbf{E}\{V(k+1,x_{1}(k+1),x_{2}(k+1),\vartheta(k+1))\\
&\ \ \ \ -V(k,x_{1}(k),x_{2}(k),\vartheta(k))\}\\
&\leq-(\xi_{1}-\zeta_{0})\mathbf{E}\{
\|x_{1}(k)\|^{2}_{\mathbb{R}^{n}}\}\\
&\ \ \ -(\zeta_{2}+\zeta_{0}^{-1}\zeta_{1}^{2})\mathbf{E}\{\|x_{2}(k)\|^{2}_{\mathbb{R}^{n}}\}.
\end{align*}
 We know that there exist
$\xi_3>0$ and $\xi_4>0$ such that
$\xi_3 I\leq P_{1}(\ell)\leq\|P_{1}\|_{\mathcal{H}^{n\times n}_{\infty}}I \ \mu\text{-} a.e.$
and
$\xi_4 I\leq P_{2}(\ell)\leq \|P_{2}\|_{\mathcal{H}^{n\times n}_{\infty}}I\ \mu\text{-} a.e.$ via observing $P_{1}\in\mathcal{H}^{n+*}_{\infty}$ and $P_{2}\in\mathcal{H}^{n+*}_{\infty}$.
Then it follows that
\begin{equation}\label{lyV769}
\begin{aligned}
&\mathbf{E}\left\{V(k+1,x_{1}(k+1),x_{2}(k+1),\vartheta(k+1))\right\}\\
&\leq (1-\bar{\alpha})\mathbf{E}\left\{V(k,x_{1}(k),x_{2}(k),\vartheta(k))\right\},
\end{aligned}
\end{equation}
where
$\bar{\alpha}=\min\left\{\frac{\xi_{1}-\zeta_{0}}
{\|P_{1}\|_{\mathcal{H}^{n\times n}_{\infty}}},\frac{\xi_2}{2\|P_{2}\|_{\mathcal{H}^{n\times n}_{\infty}}}\right\}.$
Moreover, it can be seen from \eqref{743P1} and \eqref{747P2} that
$\bar{\alpha}\in(0,1).$
Setting $\alpha=1-\bar{\alpha}$ and then recursively applying \eqref{lyV769}, one obtains that
\begin{align*}
&\mathbf{E}\left\{V(k,x_{1}(k),x_{2}(k),\vartheta(k))\right\}\\
&\leq\alpha^{k}\mathbf{E}\left\{V(0,x_{1}(0),x_{2}(0),\vartheta(0))\right\}\\
&\leq\bar{\beta}\alpha^{k}[\|x_{1}(0)\|_{\mathbb{R}^{n}}^{2}+\|x_{2}(0)\|_{\mathbb{R}^{n}}^{2}],
\end{align*}
where
$\bar{\beta}=\max\left\{\|P_{1}\|_{\mathcal{H}^{n\times n}_{\infty}},\frac{2(\zeta_{2}+\zeta_{0}^{-1}\zeta_{1}^{2})}{\xi_2}\|P_{2}\|_{\mathcal{H}^{n\times n}_{\infty}}
\right\}.$
This yields
$$\mathbf{E}\left\{\|x_{1}(k)\|^{2}_{\mathbb{R}^{n}}
+\|x_{2}(k)\|^{2}_{\mathbb{R}^{n}}\right\}
\leq \beta\alpha^{k}[\|x_{1}(0)\|_{\mathbb{R}^{n}}^{2}
+\|x_{2}(0)\|_{\mathbb{R}^{n}}^{2}],
$$
where $\beta=\frac{\bar{\beta}}{\hat{\beta}},$ $\hat{\beta}=\min\{\xi_{3},\frac{2(\zeta_{2}
+\zeta_{0}^{-1}\zeta_{1}^{2})\xi_4}{\xi_2}\},$
 and  $\alpha\in(0,1).$
Thus, system \eqref{check736} is EMSS, and the desired result is obtained.
\end{proof}

\begin{remark}
Theorem \ref{smg} can be viewed as an extended version of the small gain theorem established in \cite{Aberkane2015SIAM}, specifically tailored for MJLSs with the Markov chain on a Borel space. Moreover, Theorem \ref{smg} is still valid for the case where the output of system \eqref{575systemg1} is constructed as $z_{1}(k)=C_{1}(\vartheta(k))x_{1}(k)+D_{1}(\vartheta(k))v_{1}(k).$ In this scenario, by Theorem A.4.9 in \cite{Curtain1995}, $(\mathcal{I}-D_{1}D_{2})^{-1}\in{\mathcal{H}^{r\times r}_{\infty}}$ is naturally satisfied, as guaranteed by assumption $(ii)$ of Theorem \ref{smg}.
\end{remark}

\section{Stability Radii}\label{SecStabilityRadii}
In this section, we will study the stability radius of uncertain MJLSs, which plays the role of the robustness measure of stability. It will be shown that a lower bound of the stability radius can be derived as an application of the established small gain theorem.

Consider the system with parametric uncertainties as follows:
\begin{equation}\label{disturb906}
x(k+1)=[A(\vartheta(k))+B(\vartheta(k))\Delta(\vartheta(k)) C(\vartheta(k))]x(k),
\end{equation}
$\ k\in\mathbb{N},$
where $B, C$ model the available information about
the structure and scale of the uncertainty, and $\Delta$ represents the unknown part of
the admissible uncertainty, that is, $\Delta=\{\Delta(\ell)\}_{\ell\in\Theta}\in\mathcal{H}^{r\times r}_{\infty}.$ In this section, the nominal system $(A|\mathbb{G})$ is assumed to be EMSS. System \eqref{disturb906} can be viewed as the perturbed model of the nominal system  $(A|\mathbb{G}),$  and we are interested in the range of the unknown perturbation $\Delta$ within which the system \eqref{disturb906}
can still maintain satisfactory stability.
To this end, the concept of the stability radius is put forward, interpreted as the norm of the maximum perturbation $\Delta$ such that the perturbed systems are stable. Next, the formal definition of the stability radius is presented.
\begin{definition}
The stability radius of $(A|\mathbb{G})$ w.r.t. parametric uncertainties as in \eqref{disturb906}, is
$
r(A;B,C|\mathbb{G})=\inf\{\|\Delta\|_{\mathcal{H}^{r\times r}_{\infty}}|\Delta\in\mathcal{H}^{r\times r}_{\infty}, \text{ \eqref{disturb906} is not EMSS}\}.
$
\end{definition}

Consider MJLS \eqref{414SGsystem} with $C=\{\Delta(\ell)\}_{\ell\in\Theta}\in\mathcal{H}^{r\times n}_{\infty}$ defined as $C(\ell)=0,\ \ell\in\Theta.$ The resulting input-output operator is
$\mathfrak{L}_{C=0}(v)(k)=D(\vartheta(k))v(k),\ k\in\mathbb{N},$
whose $H_{\infty}$ norm can be calculated as in Proposition \ref{928radpro}.

\begin{proposition}\label{928radpro}
If $(A|\mathbb{G})$ is EMSS, then $\|\mathfrak{L}_{C=0}\|_{\infty}=\|D\|_{\mathcal{H}^{r\times r}_{\infty}}$.
\end{proposition}
\begin{proof}
$\|\mathfrak{L}_{C=0}\|_{\infty}\leq \|D\|_{\mathcal{H}^{r\times r}_{\infty}}$ can be obtained by that
$\|\mathfrak{L}_{C=0}(v)\|_{l^{2}(\mathbb{N};\mathbb{R}^{r})}
\leq \|D\|_{\mathcal{H}^{r\times r}_{\infty}}{\|v\|_{l^{2}(\mathbb{N};\mathbb{R}^{r})}}$
holds for any $\vartheta_{0}$ and nonzero input  $v \in l^{2}(\mathbb{N};\mathbb{R}^{r}).$
We assert that $\|\mathfrak{L}_{C=0}\|_{\infty}< \|D\|_{\mathcal{H}^{r\times r}_{\infty}}$ does not hold. Otherwise, according to Lemma \ref{lemmaBRL}, since $(A|\mathbb{G})$ is EMSS, then $P=0$ is the  unique stabilizing solution of \eqref{lmiARE}, which yields that for some $\eta>0$,
$D(\ell)^{T}D(\ell)\leq \|D\|_{\mathcal{H}^{r\times r}_{\infty}}^{2}I-\eta I \ \mu\text{-}a.e..$
This leads to
$\|D\|_{\mathcal{H}^{r\times r}_{\infty}}^{2}\leq \|D\|_{\mathcal{H}^{r\times r}_{\infty}}^{2}-\eta$
with $\eta>0$.
It is apparently a contradiction, and the desired result is established.
\end{proof}
\begin{remark}
In Proposition 3.4 of \cite{Aberkane2015SIAM}, a result similar to Proposition \ref{928radpro} was established for finite MJLSs.
However, during the proof of Proposition \ref{928radpro}, the technique employed in \cite{Aberkane2015SIAM} is limited due to the uncountable nature of the considered Borel space. Therefore, we have adopted a proof by contradiction to achieve the desired result.
\end{remark}

In the following theorem, a lower bound of $r(A;B,C|\mathbb{G})$ is given.
\begin{theorem}\label{lowboundsta963}
Suppose that $(A|\mathbb{G})$ is EMSS.
Consider MJLS
\begin{equation}\label{D0973}
\left\{
\begin{array}{ll}
x(k+1)=A(\vartheta(k))x(k)+B(\vartheta(k))v(k), \\
z(k)=C(\vartheta(k))x(k),\ k\in\mathbb{N}.
\end{array}
\right.
\end{equation}
Its input-output operator is denoted by $\mathfrak{L}_{D=0}$. Then, $r(A;B,C|\mathbb{G})\geq \|\mathfrak{L}_{D=0}\|_{\infty}^{-1}.$
\end{theorem}
\begin{proof}
We view  \eqref{D0973} and system $z_{2}=\Delta(\vartheta(k))v_{2}(k)$ as systems \eqref{575systemg1} and \eqref{583systemg2}, respectively.
The proof can be completed by combining Theorem \ref{smg} and Proposition \ref{928radpro}.
\end{proof}
\begin{remark}
Another way to interpret Theorem \ref{lowboundsta963} is
from the perspective of the output feedback control. To be precise, note that one can write the uncertain system \eqref{disturb906} as
system \eqref{D0973} with control $v(k)=\Delta(\vartheta(k))z(k),\  k\in\mathbb{N}.$
Theorem \ref{lowboundsta963} implies that system \eqref{D0973} remains stable with $v(k)=\Delta(\vartheta(k))z(k),\  k\in\mathbb{N}$ when $\Delta$ satisfies $\|\Delta\|_{\mathcal{H}^{r\times r}_{\infty}}\leq\|\mathfrak{L}_{D=0}\|_{\infty}^{-1}.$
Certainly, to achieve the implementation of this mode-dependent controller,  $(ii)$ in Assumption \ref{Assumption1} is needed.
\end{remark}

\section{Robust Stability}\label{SecRobustStability}
This section will continue to focus on the robustness of system \eqref{disturb906}. With the help of Theorem \ref{lowboundsta963}, we get that, under the hypothesis of $(A|\mathbb{G})$ being EMSS, the uncertain system \eqref{disturb906} remains stable whenever the uncertainty $\Delta$ satisfies $\|\Delta\|_{\mathcal{H}^{r\times r}_{\infty}} \leq \|\mathfrak{L}_{D=0}\|_{\infty}^{-1}$. Unfortunately, determining the value of $\|\mathfrak{L}_{D=0}\|_{\infty}$ is not a straightforward task.
An alternative approach to addressing this problem, which may introduce conservatism, is as follows:
When $(A|\mathbb{G})$ is EMSS, imposing the constraint $\|\mathfrak{L}_{D=0}\|_{\infty}< \sqrt{\gamma}$ with $\gamma>0$, we find that $(\sqrt{\gamma})^{-1} < \|\mathfrak{L}_{D=0}\|_{\infty}^{-1}$; as long as the uncertainty $\Delta$ satisfies $\|\Delta\|_{\mathcal{H}^{r\times r}_{\infty}}\leq (\sqrt{\gamma})^{-1}$, it can be concluded that the uncertain system \eqref{disturb906} is EMSS. Due to this, given $ \gamma>0$,
we assume that the considered admissible uncertainties satisfy
the following form:
\begin{equation}\label{eqUncertaintyCon}
\|\Delta\|_{\mathcal{H}^{r\times r}_{\infty}}\leq (\sqrt{\gamma})^{-1}.
\end{equation}
Next, give the notion of robust stability, which is inspired by Definition 2.1 in \cite{todorov2012new}.
\begin{definition}\label{DefRobuststability}
Given $\gamma>0$, system \eqref{disturb906} is said to be robustly stable whenever it is EMSS, regardless of $\Delta$ as in \eqref{eqUncertaintyCon}.
\end{definition}

To proceed with investigating the robust stability of \eqref{disturb906}, according to the above discussion, of primary interest will be establishing the feasible conditions for ensuring $(A|\mathbb{G})$ being EMSS and $\|\mathfrak{L}_{D=0}\|_{\infty}<\sqrt{\gamma}$. Fortunately, Lemma \ref{lemmaBRL} provides two alternative options. However,
as mentioned in \cite{Xiao2023}, determining the existence of the stabilizing solution to the coupled-AREs \eqref{lmiARE} remains challenging. Here, we pick another equivalent way: Verify the existence of a uniformly positive solution $P\in\mathcal{H}_{\infty}^{n+*}$ to the inequality \eqref{LMI741} that satisfies the sign condition $\Psi_{3}^{\gamma}(P)(\ell)\gg 0\ \mu\text{-}a.e..$ This can be formally expressed as follows:

\begin{proposition}\label{pro1091}
Given $\gamma>0$, $(A|\mathbb{G})$ is EMSS and system \eqref{D0973} satisfies $\|\mathfrak{L}_{D=0}\|_{\infty}<\sqrt{\gamma}$ iff there exists $P\in\mathcal{H}_{\infty}^{n+*}$ satisfying the  following inequality:
\begin{equation}\label{1877noLinearEP}
\Xi(\ell)\gg0 \ \mu\text{-}a.e.,
\end{equation}
where
\begin{equation}\label{xi816}
\Xi(\ell)=\left[
  \begin{array}{cccc}
    \mathcal{E}(P)(\ell)^{-1} & 0 & A(\ell) & B(\ell)\\
    0 & I &  C(\ell) & 0 \\
   A(\ell)^{T} & C(\ell)^{T} & P(\ell) & 0 \\
   B(\ell)^{T} & 0 & 0 & \gamma I \\
  \end{array}
\right].
\end{equation}
\end{proposition}
\begin{proof}
The result can be inferred by applying Lemma \ref{lemmaBRL} to system \eqref{D0973} and then using Proposition \ref{Schurcom}.
\end{proof}

\subsection{Finite MJLSs}\label{SecFiniteMJLS}
The discussion in this subsection tailors to the robustness of finite MJLSs. In this special case, $\|\mathfrak{L}_{D=0}\|_{\infty}$ can be easily determined by solving a linear programming problem
constrained by certain LMIs.
To be specific, adapting proposition \ref{pro1091} to make it applicable for finite MJLSs, we can derive the next result as a corollary of Theorem \ref{lowboundsta963}:
\begin{proposition}\label{finiteMarkovchain135}
Consider the scenario of the Markov chain $\{\vartheta(k), k\in\mathbb{N}\}$ taking values in a finite set $\Theta_{N}=\overline{1,N}$ with
the initial distribution given by
$\pi_{i}=\mathbb{P}(\vartheta_0=i),\ i\in\Theta_{N}$ and
the transition probability matrix $[p_{ij}]_{N\times N}$
given by $p_{ij}=
\mathbb{P}(\vartheta(k+1)=j|\vartheta(k)=i),\ i\in\Theta_{N},\
j\in\Theta_{N}.$
Regarding the system \eqref{disturb906}, the notions
$A(i):=A(\vartheta(k)=i),\ B(i):=B(\vartheta(k)=i)$, $C(i):=C(\vartheta(k)=i),$ and $\Delta(i):=\Delta(\vartheta(k)=i)$ are adopted for any $i\in\Theta_{N}.$ In this case,
the uncertain system \eqref{disturb906} is EMSS when the uncertainty $\Delta$ satisfies
$$\|\Delta\|_{\max}:=\max_{i\in\Theta_{N}}\|\Delta(i)\|
_{\mathbb{R}^{r\times r}}\leq \|\mathfrak{L}_{D=0}\|_{\infty}^{-1},$$
 where
$\|\mathfrak{L}_{D=0}\|_{\infty}$ can be calculated from the optimal solution of a convex programming problem expressed as follows:
\begin{equation*}
\|\mathfrak{L}_{D=0}\|_{\infty}^{2}=\inf_{(\gamma, P)\in \mathbb{U}} \gamma.
\end{equation*}
Here, $\gamma>0$,
$P=\{P_{i}\}_{i\in\Theta_{N}}$ with  $P_{i}$ being a positive definite matrix for any $i\in\Theta_{N}$, and
$\mathbb{U}$ is the set of $(\gamma, P)$ satisfying the following LMIs:
\begin{equation}\label{finite1125eq}
\left[
  \begin{array}{cccc}
    P_{sum}(i) & 0 & {P_{sum}(i)}A(i) & {P_{sum}(i)}B(i)\\
    0 &  I &  C(i) & 0 \\
   A(i)^{T}{P_{sum}(i)} &  C(i)^{T} & P_{i} & 0 \\
   B(i)^{T}{P_{sum}(i)} & 0 & 0 &  \gamma I \\
  \end{array}
\right]
>0,
\end{equation}
where $P_{sum}(i)=\Sigma_{j=1}^{N}p_{ij}P_{j},\ i\in\Theta_{N}.$
\end{proposition}
\begin{remark}
For finite MJLSs, the above LMI method is both effective and easy to implement.
Specifically, the LMI constraints \eqref{finite1125eq} involve $N$ LMIs with $N$ matrix variables, where $N$ is the cardinality of the state set of the Markov chain. Additionally, the coupling term $\Sigma_{j=1}^{N} p_{ij} P_{j}$ is a linear combination of $N$ matrix variables.
These allow the convex programming problem to be solved using LMI solvers.
\end{remark}

\subsection{ MJLS with  Continuous-Valued Markov Jump Parameters}\label{SecRobustofConCase}
In Subsection \ref{SecFiniteMJLS}, a lower bound of the stability radius for the finite MJLS is determined by the optimal solution of a convex programming problem, in which the constraint functions are LMIs involving a finite number of matrix variables. The key to obtaining this result lies in the fact that, for the finite MJLS, the inequality \eqref{1877noLinearEP} can be deduced to solvable LMIs \eqref{finite1125eq}. Nevertheless, under the general setting of the Markov chain taking values in a Borel set, further investigation is needed to establish the conditions for the feasibility of the inequality \eqref{1877noLinearEP}.
Basically, this general setting not only covers the special case of the Markov chain taking values in a finite set but also includes the case where the Markov chain takes values in a continuous set. As we will see, the latter case, which is  the focus of this subsection,
presents some difficulties to be circumvented.

Assume that the considered state set of the Markov chain is a finite interval $\Theta_{[a,b]}:=[a,b],$ where $a<b\in\mathbb{R}^{1}.$ It is known that $\Theta_{[a,b]}$ has a cardinality of $c=2^{\aleph_{0}}$.
In this case, inequality \eqref{1877noLinearEP} unavoidably involves infinite (in fact uncountably infinite) inequalities as well as matrix variables. Additionally, the coupling term $\int_{\Theta}g(\ell,t)P(t)\mu(dt)$, which involves the integration of uncountably infinite matrix variables, poses certain challenges that need to be addressed.
Consequently, effective numerical methods are needed to handle these infinite inequalities involving uncountably many matrix variables. To this end, a gridding method will be introduced in the following, as previously carried out in \cite{Costa2015, Costa2016,Costa2017} and fully discussed in \cite{Masashi2018}. More formally,
by implementing grid partitioning to the state set $\Theta_{[a,b]}$, we divide the interval into a finite number of disjoint subintervals, and the originally infinite-dimensional problem can be converted into a finite-dimensional one.

We begin by selecting
$N+1$  grid points $h_0< h_1<\ \dots<h_N$ to form
a partition $H=\{h_0, h_1,\cdots, h_N\}$  of $\Theta_{[a,b]}$ such that $h_0 = a$ and $h_N = b$. This partition leads to the formation of a set of subintervals
$$
\Lambda_{i}:=
\begin{cases}
[h_{i-1}, h_i), &  i\in{\overline{1, N-1}}, \\
[h_{N-1}, h_N],& i=N.
\end{cases}
$$
Next, for each $i\in\overline{1, N},$ select a discrete point $\bar{h}_{i}\in{\Lambda_{i}}$,  e.g., $\bar{h}_{i}$ is the center of $\Lambda_{i}.$
For simplicity, in what follows, denote
$
q_{i}(\ell):=\int_{\Lambda_{i}}g(\ell,t)\mu(dt),
$
$$
Q(\ell):=\left[
          \begin{array}{cccc}
            \sqrt{q_{1}(\ell)}I & \sqrt{q_{2}(\ell)}I & \cdots & \sqrt{q_{N}(\ell)}I \\
          \end{array}
        \right], \ \ell\in\Theta_{[a,b]},
$$
and $A_{i}:=A(\bar{h}_{i}),\ B_{i}:=B(\bar{h}_{i}),\ C_{i}:=C(\bar{h}_{i}),\ Q_{i}:=Q(\bar{h}_{i}),\ i\in\overline{1, N}.$
Regarding the defined function $Q$, one can verify that
$Q\in{\mathcal{H}^{n\times (nN)}_{\infty}}.$
To use the gridding technique, we make the following assumption.
\begin{assumption}\label{ass1139finitegrid}
For the given $\{\Lambda_{i},\ i\in\overline{1, N}\}$ and $\{\bar{h}_{i},\ i\in\overline{1, N}\},$
there exist scalars $\sigma_{A,i}>0,\ \sigma_{B,i}>0,\ \sigma_{C,i}>0$ and $\sigma_{Q,i}>0$ such that
\begin{align}
&\|A(\ell)-A_{i}\|_{\mathbb{R}^{n\times n}}
\leq \sigma_{A,i},
\|B(\ell)-B_{i}\|_{\mathbb{R}^{n\times r}}\leq \sigma_{B,i}, \label{assABCQ} \\
&\|C(\ell)-C_{i}\|_{\mathbb{R}^{r\times n}}\leq \sigma_{C,i},
\|Q(\ell)-Q_{i}\|_{\mathbb{R}^{n\times (nN)}}\leq \sigma_{Q,i}
\nonumber
\end{align}
hold for $\mu$-almost all $\ell\in \Lambda_{i}$ and $i\in\overline{1, N}$.
\end{assumption}

In Proposition \ref{promain}, we will show that the feasibility of a particular LMI problem can guarantee the existence of a uniformly positive solution to the inequality \eqref{1877noLinearEP}. The exact formulation of the LMI feasibility problem (LMI-Problem) is given as follows:

\begin{definition}\label{LMIfeasP}
Given scalars $\gamma>0$, $\sigma_{A,i}>0,\ \sigma_{B,i}>0,\ \sigma_{C,i}>0$ and $\sigma_{Q,i}>0$, the LMI-Problem consists of finding, if possible, nonsingular matrices $X_{i}\in\mathbb{R}^{n\times n}$, positive definite matrices $P_{i}\in\mathbb{S}^{n+*}$,
as well as scalars $\alpha_{i}>0$, $\beta_{i}>0$, and $\rho_{i}>0$ satisfying the following LMIs:
\begin{equation}\label{1171PLMI}
\Pi^{1}_{i}>0,\ i\in\overline{1, N},
\end{equation}
where $\Pi^{1}_{i}$ is given by \eqref{Pi1907} with  $\Upsilon=\diag\{P_1, P_2, \cdots,  P_{N}\}.$
\end{definition}
\begin{figure*}
\small
\vspace{-5pt}
\begin{equation}\label{Pi1907}
\begin{split}
\Pi^{1}_{i}:=\left[
   \begin{array}{cccccccccccccc}
     X_{i}+X_{i}^{T} & 0 & X_{i}A_{i}  & X_{i}B_{i} & Q_{i}\Upsilon  & 0 & 2\sigma_{A,i}X_{i} & 0 & \sigma_{B,i}X_{i} & 0 & 0 & 0 & \rho_{i}I & 0\\
     0 & I & C_{i} & 0 & 0 & 2\sigma_{C,i}I & 0 & 0 & 0 & 0 & 0 & 0 & 0 & \rho_{i}I \\
      A_{i}^{T}X_{i}^{T} & C_{i}^{T} &  P_{i}  & 0  & 0 & 0 & 0 & \alpha_{i} I & 0 & 0 & 0 & 0 & 0 & 0\\
     B_{i}^{T}X_{i}^{T} &  0 &  0     & \gamma I & 0 & 0 & 0 & 0 & 0 & 0 & \beta_{i}I & 0 & 0 & 0\\
    \Upsilon^{T}Q_{i}^{T} & 0 & 0 & 0     &  \Upsilon  & 0 & 0 & 0 & 0 & 0 & 0 &  \sigma_{Q,i}\Upsilon & 0 & 0\\
     0 & 2\sigma_{C,i}I & 0 & 0 & 0 &  2\alpha_{i} I & 0 & 0 & 0 & 0 & 0 & 0 & 0 & 0\\
     2\sigma_{A,i}X_{i}^{T} & 0 & 0 & 0 & 0 & 0 &  2\alpha_{i} I & 0 & 0 & 0 & 0 & 0 & 0 & 0\\
     0 & 0 & \alpha_{i}I & 0 & 0 & 0 & 0 & \alpha_{i} I & 0 & 0 & 0 & 0 & 0 & 0\\
     \sigma_{B,i}X_{i}^{T} & 0 & 0 & 0 & 0 & 0 & 0 & 0 & \beta_{i} I & 0 & 0 & 0 & 0 & 0 \\
     0 & 0 & 0 & 0 & 0 & 0 & 0 & 0 & 0 & \beta_{i} I & 0 & 0 & 0  & 0\\
     0 & 0 & 0 & \beta_{i} I & 0 & 0 & 0 & 0 & 0 & 0 & \beta_{i}I & 0 & 0 & 0\\
     0 & 0 & 0 & 0 & \sigma_{Q,i} \Upsilon & 0 & 0 & 0 & 0 & 0 & 0 & \rho_{i}I& 0 & 0 \\
\rho_{i}I & 0 & 0 & 0 & 0 & 0 & 0 & 0 & 0 & 0 & 0 & 0 & \rho_{i}I  & 0\\
 0 & \rho_{i}I  & 0 & 0 & 0 & 0 & 0 & 0 & 0 & 0 & 0 & 0 & 0  & \rho_{i}I\\
\end{array}
\right]
\end{split}
\end{equation}
\vspace{-5pt}
\end{figure*}

\begin{proposition}\label{promain}
Suppose that Assumption \ref{ass1139finitegrid} is fulfilled.
If the LMI-Problem stated in Definition \ref{LMIfeasP} is feasible,
then there exists a solution
$P=\{P(\ell)\}_{\ell\in\Theta}\in\mathcal{H}_{\infty}^{n+*}$ to \eqref{1877noLinearEP}.
\end{proposition}
\begin{proof}
Assume that the LMI-Problem is feasible with
$P_{i}\in\mathbb{S}^{n+*}$ and nonsingular matrices
$X_{i}\in\mathbb{R}^{n\times n},\ i\in\overline{1, N}.$
Define piecewise constant matrix-valued functions
$P\in\mathcal{H}_{\infty}^{n+*}$ and  $X\in\mathcal{H}_{\infty}^{n\times n}$ by $P(\ell)=P_{i}$ and $X(\ell)=X_{i}$ for any $\ell\in\Lambda_{i}$ and $i\in\overline{1, N}$. We will show that, under the given conditions,
\eqref{1877noLinearEP} holds with the defined $P\in\mathcal{H}_{\infty}^{n+*}$.
To handle the nonlinear matrix function $\mathcal{E}(P)(\ell)^{-1}$ in \eqref{1877noLinearEP}, we consider the non-convex inequality of that
$$X_{i}\mathcal{E}(P)(\ell)^{-1}X_{i}^{T}\geq X_{i}+X_{i}^{T}-\mathcal{E}(P)(\ell)$$
holds for $\mu$-almost all $\ell\in\Lambda_{i}$ and $i\in\overline{1, N}$. Due to $\mathcal{E}(P)\in\mathcal{H}_{\infty}^{n+*}$, this can be directly derived from
that
$$[X(\ell)-\mathcal{E}(P)(\ell)]\mathcal{E}(P)(\ell)^{-1}
[X(\ell)-\mathcal{E}(P)(\ell)]^{T}\geq 0  \ \mu\text{-}a.e..$$
Hence, it is established that
\begin{equation}\label{1188diag}
\begin{split}
\diag\{X_{i},I, I, I\}\cdot
\Xi(\ell)\cdot
\diag\{X_{i}^{T},I, I, I\}
\geq\Xi^{1}_{i}(\ell)
\end{split}
\end{equation}
holds for $\mu$-almost all $\ell\in\Lambda_{i}$ and $i\in\overline{1, N},$ where
$\Xi(\ell)$ is given in \eqref{xi816} and
$$\Xi^{1}_{i}(\ell)=\left[
  \begin{array}{cccc}
     X_{i}+ X_{i}^{T}-\mathcal{E}(P)(\ell) & 0 & X_{i}A(\ell) & X_{i}B(\ell)\\
    0 & I &  C(\ell) & 0 \\
   A(\ell)^{T}X_{i}^{T} & C(\ell)^{T} & P_{i} & 0 \\
   B(\ell)^{T}X_{i}^{T} & 0 & 0 & \gamma I \\
  \end{array}
\right].
$$
Note that \eqref{1877noLinearEP} is equivalent to that for any $i\in\overline{1, N}$, $\diag\{X_{i},I, I, I\}\cdot\Xi(\ell)\cdot\diag\{X_{i}^{T},I, I, I\}\gg0$ holds for $\mu$-almost all $\ell\in\Lambda_{i}$.
Using the identity
$$\mathcal{E}(P)(\ell)=\Sigma_{j=1}^{N}q_{j}(\ell)P_{j}
=Q(\ell)\Upsilon Q(\ell)^{T} \ \mu\text{-}a.e.,$$
and invoking Proposition \ref{Schurcom}, we find that
$\Xi^{1}_{i}(\ell)\gg0$ equals to
$ \Xi^{2}_{i}(\ell)\gg 0,$ where
\begin{equation*}\label{1267xi2}
 \Xi^{2}_{i}(\ell)=\left[
  \begin{array}{ccccc}
     X_{i}+ X_{i}^{T} & 0 & X_{i}A(\ell) & X_{i}B(\ell)& Q(\ell)\Upsilon\\
    0 & I &  C(\ell) & 0 &0 \\
   A(\ell)^{T}X_{i}^{T} & C(\ell)^{T} & P_{i} & 0 &0 \\
   B(\ell)^{T}X_{i}^{T} & 0 & 0 & \gamma I & 0\\
    \Upsilon Q(\ell)^{T} & 0 & 0 & 0 & \Upsilon\\
  \end{array}
\right].
\end{equation*}
Therefore, to obtain \eqref{1877noLinearEP},
by recalling \eqref{1188diag},
we next need to prove that
$\Xi^{2}_{i}(\ell)\gg0$ holds for $\mu$-almost all $\ell\in\Lambda_{i}$ and $i\in\overline{1, N}.$
Since Assumption \ref{ass1139finitegrid} is fulfilled, for any $i\in\overline{1, N},$ there exist
$\Gamma_{A,i}\in{\mathcal{H}^{n\times n}_{\infty}},\ \Gamma_{B,i}\in{\mathcal{H}^{n\times r}_{\infty}},\
\Gamma_{C,i}\in{\mathcal{H}^{r\times n}_{\infty}},$ $\Gamma_{Q,i}\in{\mathcal{H}^{n\times (nN)}_{\infty}}$
defined as
$\Gamma_{A,i}(\ell)=[A(\ell)-A_{i}]/\sigma_{A,i},$
$\Gamma_{B,i}(\ell)=[B(\ell)-B_{i}]/\sigma_{B,i},$
$\Gamma_{C,i}(\ell)=[C(\ell)-C_{i}]/\sigma_{C,i},$
$\Gamma_{Q,i}(\ell)=[Q(\ell)-Q_{i}]/\sigma_{Q,i},\ \ell\in \Lambda_{i}$
satisfying
\begin{equation}\label{xishu1261}
\begin{split}
&\|\Gamma_{A,i}\|_{\mathcal{H}^{n\times n}_{\infty}}\leq 1,\ \|\Gamma_{B,i}\|_{\mathcal{H}^{n\times r}_{\infty}}\leq 1,\\
&\|\Gamma_{C,i}\|_{\mathcal{H}^{r\times n}_{\infty}}\leq 1,\ \|\Gamma_{Q,i}\|_{\mathcal{H}^{n\times (nN)}_{\infty}}\leq 1.
\end{split}
\end{equation}
In addition, for any $i\in\overline{1, N}$ and $\ell\in \Lambda_{i}$,
it holds that
\begin{equation*}
\begin{split}
&\|\left[
    \begin{array}{c}
      \Gamma_{A,i}(\ell) \\
      \Gamma_{C,i}(\ell) \\
    \end{array}
  \right]\left[
    \begin{array}{cc}
      \Gamma_{A,i}(\ell)^{T} & \Gamma_{C,i}(\ell)^{T} \\
    \end{array}
  \right]\|_{\mathbb{R}^{(n+r)\times (n+r)}}\\
&\leq
\lambda_{\max}\{\Gamma_{A,i}(\ell)^{T}\Gamma_{A,i}(\ell)
+\Gamma_{C,i}(\ell)^{T}\Gamma_{C,i}(\ell)\}\\
&\leq  \lambda_{\max}\{\Gamma_{A,i}(\ell)^{T}\Gamma_{A,i}(\ell)\}
+\lambda_{\max}\{\Gamma_{C,i}(\ell)^{T}\Gamma_{C,i}(\ell)\}\\
&\leq  \|\Gamma_{A,i}\|_{\mathcal{H}^{n\times n}_{\infty}}^{2}+
\|\Gamma_{C,i}\|_{\mathcal{H}^{r\times n}_{\infty}}^{2}
\leq 2,
\end{split}
\end{equation*}
where the second inequality is built by Weyl's inequality (for example, see Theorem 4.3.1 in \cite{BookHorn2013}) for eigenvalues.
Then, from \eqref{xishu1261}, for any scalars $\rho_{A,C,i}>0,\ \rho_{B,i}>0$, and $\rho_{i}>0$, we can establish that
$
\Xi^{3}_{i}(\ell)\geq 0
$
holds for $\mu$-almost all $\ell\in\Lambda_{i}$ and  $i\in\overline{1, N},$
where
$$\Xi^{3}_{i}(\ell):=\rho_{A,C,i}\Xi^{4}_{i}(\ell)
+\rho_{B,i}\Xi^{5}_{i}(\ell)+\rho_{i}\Xi^{6}_{i}(\ell)$$ with
\begin{align*}
&\Xi^{4}_{i}(\ell):=\\
&\left[
              \begin{array}{cc}
                \sigma_{A,i}X_{i} & 0 \\
                0 & \sigma_{C,i}I \\
              \end{array}
            \right]
\left[
2I-\left[
    \begin{array}{c}
      \Gamma_{A,i}(\ell) \\
      \Gamma_{C,i}(\ell) \\
    \end{array}
  \right]
  \left[
    \begin{array}{c}
      \Gamma_{A,i}(\ell) \\
      \Gamma_{C,i}(\ell) \\
    \end{array}
  \right]^{T}
\right]\\
&
\cdot \left[
   \begin{array}{cc}
    \sigma_{A,i}X_{i}^{T} & 0 \\
     0 & \sigma_{C,i}I \\
   \end{array}
   \right],\\
\end{align*}
\begin{align*}
&\Xi^{5}_{i}(\ell):=\\
&\left[
              \begin{array}{cc}
                \sigma_{B,i}X_{i} & 0 \\
                0 & 0 \\
              \end{array}
            \right]
\left[
I-\left[
    \begin{array}{c}
      \Gamma_{B,i}(\ell) \\
      0 \\
    \end{array}
  \right]
  \left[
    \begin{array}{c}
      \Gamma_{B,i}(\ell) \\
      0 \\
    \end{array}
  \right]^{T}
\right]
\left[
   \begin{array}{cc}
    \sigma_{B,i}X_{i}^{T} & 0 \\
     0 & 0 \\
   \end{array}
   \right],
\end{align*}
\begin{align*}
&
\Xi^{6}_{i}(\ell):=
I-\left[
    \begin{array}{c}
      \Gamma_{Q,i}(\ell) \\
       0 \\
    \end{array}
  \right]
  \left[
    \begin{array}{cc}
      \Gamma_{Q,i}(\ell)^{T} & 0 \\
    \end{array}
  \right].
\end{align*}
Denote $\alpha_{i}:=\rho_{A,C,i}^{-1}$ and $\beta_{i}:=\rho_{B,i}^{-1}, i\in\overline{1, N}$.
Observe that
$$
\Xi^{2}_{i}(\ell)-\diag\{\Xi^{3}_{i}(\ell),0 ,0,0\}
=
\left[
  \begin{array}{cc}
    I\ \  & \Xi^{7}_{i}(\ell)^{T} \\
  \end{array}
\right]
\cdot \Pi^{1}_{i}\cdot
\left[
  \begin{array}{c}
    I \\
    \Xi^{7}_{i}(\ell) \\
  \end{array}
\right],
$$
where
\begin{equation*}\label{Xi71364}
\Xi^{7}_{i}(\ell)=\left[
  \begin{array}{ccccc}
   0 & -\alpha_{i}^{-1}\sigma_{C,i}I & 0 & 0 & 0 \\
   -\alpha_{i}^{-1}\sigma_{A,i}X_{i}^{T} & 0 & 0 & 0 & 0 \\
   \alpha_{i}^{-1}\sigma_{A,i}\Gamma_{A,i}(\ell)^{T}X_{i}^{T} & \alpha_{i}^{-1}\sigma_{C,i}\Gamma_{C,i}(\ell)^{T} & 0 & 0 & 0 \\
   -\beta_{i}^{-1}\sigma_{B,i}X_{i}^{T} & 0 & 0 & 0 & 0 \\
    0 & 0 & 0 & 0 & 0 \\
   \beta_{i}^{-1}\sigma_{B,i}\Gamma_{B,i}(\ell)^{T}X_{i}^{T} & 0 & 0 & 0 & 0 \\
    \Gamma_{Q,i}(\ell)^{T} & 0 & 0 & 0 & 0 \\
    -I & 0 & 0 & 0 & 0 \\
    0 & -I & 0 & 0 & 0 \\
  \end{array}
\right].
\end{equation*}
It can be verified that if $\Pi^{1}_{i}>0$, then
$$\left[
  \begin{array}{cc}
    I\ \  & \Xi^{7}_{i}(\ell)^{T} \\
  \end{array}
\right]
\cdot \Pi^{1}_{i}\cdot
\left[
  \begin{array}{c}
    I \\
    \Xi^{7}_{i}(\ell) \\
  \end{array}\right]
$$
holds for $\mu$-almost all $\ell\in\Lambda_{i}$ and $i\in\overline{1, N}$.
Since $\Xi^{3}_{i}(\ell)\geq 0$ holds for $\mu$-almost all $\ell\in \Lambda_{i}$,
hence $\Pi^{1}_{i}>0$ implies that $\Xi^{2}_{i}(\ell)\gg 0$
holds for $\mu$-almost all $\ell\in\Lambda_{i}$ and $i\in\overline{1, N}$,
thereby concluding the proof.
\end{proof}

\begin{remark}
According to Proposition \ref{pro1091}, it is clear that in
Proposition \ref{promain}, an LMI-based BRL has been established for the case where the Markov chain takes values in a finite interval.
\end{remark}

Following this, we provide a sufficient condition for the robust stability of  \eqref{disturb906} in terms of LMIs.

\begin{theorem}\label{finitegrid1166}
Suppose that Assumption \ref{ass1139finitegrid} is fulfilled. Given $ \gamma>0$, if the LMI-Problem stated in Definition \ref{LMIfeasP} is feasible, then the uncertain system \eqref{disturb906} is robustly stable.
\end{theorem}
\begin{proof}
From Proposition \ref{promain}, if the given LMI-Problem is feasible, there must exist
$P\in\mathcal{H}_{\infty}^{n+*}$ such that \eqref{1877noLinearEP} is satisfied.
Then, according to Proposition \ref{pro1091} and Theorem \ref{lowboundsta963}, the uncertain system \eqref{disturb906} is EMSS whenever the uncertainty $\Delta$ satisfies $\|\Delta\|_{\mathcal{H}^{r\times r}_{\infty}}\leq (\sqrt{\gamma})^{-1}$.
\end{proof}

\begin{remark}
It is noteworthy that when $\gamma$ is sufficiently large, Theorem \ref{finitegrid1166} is reduced to a result that provides a sufficient condition for the EMSS of MJLSs with continuous-valued Markov jump parameters.
\end{remark}

The next result is obtained directly from the Schur complements.
\begin{proposition}\label{1552pro}
Given $ \gamma>0$, LMIs \eqref{1171PLMI} are equivalent to
\begin{equation}\label{1534LMIs}
\Pi^{2}_{i}-
\Pi^{3}_{i}>0,\ i\in\overline{1, N},
\end{equation}
where
\begin{equation}\label{1552pi2}
\Pi^{2}_{i}=\left[
  \begin{array}{ccccc}
     X_{i}+X_{i}^{T} & 0 & X_{i}A_{i} & X_{i}B_{i} & Q_{i}\Upsilon \\
    0 & I & C_{i} & 0 & 0 \\
    A_{i}^{T}X_{i}^{T} &  C_{i}^{T} & P_{i} & 0 & 0 \\
    B_{i}^{T}X_{i}^{T} & 0 & 0 & \gamma I & 0 \\
    \Upsilon Q_{i}^{T} & 0 & 0 & 0 & \Upsilon \\
  \end{array}
\right]
\end{equation}
and
$\Pi^{3}_{i}=\diag\{
(\frac{2\sigma_{A,i}^{2}}{\alpha_{i}}
    +\frac{\sigma_{B,i}^{2}}{\beta_{i}})X_{i}X_{i}^{T}
    +\rho_{i}I,\  (\frac{2\sigma_{C,i}^{2}}{\alpha_{i}}+\rho_{i})I,\  \alpha_{i}I,\ \beta_{i}I,\ \frac{\sigma_{Q,i}^{2}}{\rho_{i}}\Upsilon^{2}\}.
    $
\end{proposition}

\begin{remark}\label{Re1338fic}
The finite MJLS in the form of \eqref{disturb906} described in
Proposition \ref{finiteMarkovchain135}
can be re-characterized by using a MJLS with
a continuous-valued Markov chain.
This can be achieved, for example, through the following way:
Denote $\Theta_{[0,N]}:=[0,N]$ and
partition $\Theta_{[0,N]}$ into a set of subintervals
\begin{equation}\label{eq1346Lam}
\Lambda_{i}:=
\begin{cases}
[i-1,i), &  i\in{\overline{1,N-1}}, \\
[N-1,N],& i=N.
\end{cases}
\end{equation}
Define a continuous-valued Markov chain $\{\vartheta(k), k\in\mathbb{N}\}$ taking values in $\Theta_{[0,N]}$
with the initial distribution given by
$\mu_{0}(\vartheta_0\in{\Lambda_{i}})=\pi_{i},\ i\in{\Theta_{N}}$
and the stochastic kernel $\mathbb{G}(\cdot,\cdot)$ satisfying
$
\mathbb{G}(\ell,\Lambda_{j})= \mathbb{P}(\vartheta(k+1)\in \Lambda_{j}|\vartheta(k)=\ell)=p_{ij} \text{ for } \ell\in \Lambda_{i},
$
$i,j\in{\Theta_{N}},$ $k\in\mathbb{N}.$
By setting $A(\vartheta(k)\in \Lambda_{i})=A(i),$ $B(\vartheta(k)\in \Lambda_{i})=B(i),$
$C(\vartheta(k)\in \Lambda_{i})=C(i)$, and $\Delta(\vartheta(k)\in \Lambda_{i})=\Delta(i)$
for any $i\in{\Theta_{N}}$ and $k\in\mathbb{N},$
we can reconstruct an uncertain system in the form of \eqref{disturb906} with the defined continuous-valued Markov chain.
Based on this approach, the robustness problem for the finite MJLSs can also be incorporated
within the framework of MJLSs with the continuous-valued Markov chain.
\end{remark}

Regarding the case of finite MJLSs, Theorem \ref{finitegrid1166} specializes as follows:

\begin{corollary}\label{1577corfinite}
Consider the case stated in Proposition \ref{finiteMarkovchain135}, where the
the Markov chain takes values in a finite set $\Theta_{N}$.
For every $i\in\Theta_{N},$
let
$A_{i}=A(i),\ B_{i}=B(i),$ $C_{i}=C(i)$ and $Q_{i}=\left[
          \begin{array}{cccc}
            \sqrt{{p_{i1}}}I & \sqrt{{p_{i2}}}I & \cdots & \sqrt{{p_{iN}}}I \\
          \end{array}
        \right].$
Given $ \gamma>0$, if the LMI-Problem stated in Definition \ref{LMIfeasP} is feasible, then the finite MJLS in the form of \eqref{disturb906} is robustly stable.
\end{corollary}
\begin{proof}
Clearly, under the partition conducted in \eqref{eq1346Lam},
assumptions \eqref{assABCQ} hold for any scalars $\sigma_{A,i}>0,\ \sigma_{B,i}>0,\ \sigma_{C,i}>0$ as well as $\sigma_{Q,i}>0$, $i\in\Theta_{N}.$
Hence, the desired result follows from Remark \ref{Re1338fic} and Theorem \ref{finitegrid1166}.
\end{proof}

The introduction of scalars $\sigma_{A,i}>0,$ $\sigma_{B,i}>0,$ $\sigma_{C,i}>0,$ and $\sigma_{Q,i}>0, i\in\Theta_{N}$ may bring some conservatism to the sufficient condition in Proposition \ref{promain}.
We will demonstrate this in the forthcoming subsection.
Regarding the case of the Markov chain taking values in a finite set, certainly, the conditions in Corollary \ref{1577corfinite} also inherit this conservatism. For this special scenario,
it is to be expected that the solutions $P_{i},\ i\in\Theta_{N}$ of the LMI-Problem for Corollary \ref{1577corfinite} can be utilized to solve the LMIs in Proposition \ref{finiteMarkovchain135}. This will be confirmed in Proposition \ref{pro1489}.
\begin{proposition}\label{pro1489}
Consider the case stated in Proposition \ref{finiteMarkovchain135}, where the
the Markov chain takes values in a finite set $\Theta_{N}$.
For every $i\in\Theta_{N},$
let
$A_{i}=A(i),\ B_{i}=B(i),$ $C_{i}=C(i)$ and $Q_{i}=\left[
          \begin{array}{cccc}
            \sqrt{{p_{i1}}}I & \sqrt{{p_{i2}}}I & \cdots & \sqrt{{p_{iN}}}I \\
          \end{array}
        \right].$
Given $ \gamma>0$, the solutions $P_{i},\ i\in\Theta_{N}$ of the LMI-Problem in Corollary \ref{1577corfinite} satisfy LMIs \eqref{finite1125eq} in Proposition \ref{finiteMarkovchain135}.
\end{proposition}
\begin{proof}
It can be obtained that $\Pi^{3}_{i}>0$ for
any scalars $\sigma_{A,i}>0,\ \sigma_{B,i}>0,\ \sigma_{C,i}>0,$ $\sigma_{Q,i}>0,$ $\alpha_{i}>0$, $\beta_{i}>0$, and $\rho_{i}>0$, where $i\in\Theta_{N}$ and $\Pi^{3}_{i}$ is defined in Proposition \ref{1552pro}.
From Proposition \ref{1552pro}
we get that the LMI constraints \eqref{1171PLMI} leads to
$\Pi^{2}_{i}>0,$ $i\in\Theta_{N}.$ On the other hand, applying the Schur complements,
one has the equivalent form of
 $\Pi^{2}_{i}>0$ as follows:
$$\left[
  \begin{array}{cccc}
     X_{i}+X_{i}^{T}-P_{sum}(i) & 0 & X_{i}A_{i} & X_{i}B_{i}  \\
    0 & I & C_{i} & 0  \\
    A_{i}^{T}X_{i}^{T} &  C_{i}^{T} & P_{i} & 0 \\
    B_{i}^{T}X_{i}^{T} & 0 & 0 & \gamma I \\
  \end{array}
\right]>0, i\in\Theta_{N},
$$
where $P_{sum}(i)=\Sigma_{j=1}^{N}p_{ij}P_{j}.$
For any $i\in\Theta_{N},$
use inequality
$X_{i}P_{sum}(i)^{-1}X_{i}^{T}\geq X_{i}+X_{i}^{T}-P_{sum}(i)$,
and an application of the congruence transformation
$\diag\{P_{sum}(i)X_{i}^{-1}, I,I, I\} \{\cdot\}\diag\{(X_{i}^{T})^{-1}P_{sum}(i),I,
I, I\}$ yields that \eqref{finite1125eq} holds with the given $\gamma.$
\end{proof}

\subsection{Feasibility and Conservatism of the Gridding Method}\label{subsecFeaConGriding}
In this subsection, we will show the feasibility and
conservatism of the proposed gridding method.
During the proof of Proposition \ref{promain}, a piecewise constant matrix-valued function was directly  employed as a solution to the inequality \eqref{1877noLinearEP}. However,
it is worth examining whether there exists such a piecewise constant matrix-valued function that meets \eqref{1877noLinearEP} under some mild technical assumptions. Besides, if the function does exist,
using this numerical approximation may lead to the resulting conditions being conservative, raising the question of how conservative the proposed gridding method actually is.
To tackle these issues, the following technical assumptions are made:
\begin{assumption}\label{continuousAss1535}
$(i)$~The measure $\mu$ on $\mathcal{B}(\Theta_{[a,b]})$ is the Lebesgue measure;\\
$(ii)$~Functions $A\in\mathcal{H}^{n\times n}_{\infty}$, $B\in\mathcal{H}^{n\times r}_{\infty}$, and
$C\in\mathcal{H}^{r\times n}_{\infty}$ are continuous;\\
$(iii)$~Function $g(\ell,t)$ is continuous on $\Theta_{[a,b]}\times\Theta_{[a,b]}.$
\end{assumption}

\begin{assumption}\label{ass1636}
Given scalar $\gamma>0$, the following conditions hold:\\
$(i)$~There exists a solution $P\in\mathcal{H}_{\infty}^{n+*}$ to \eqref{1877noLinearEP};\\
$(ii)$~There exists a sequence of matrices $\{P_{i}, i\in\overline{1, N}\}$  with $P_{i}\in\mathbb{S}^{n+*}$ such that
the piecewise constant matrix-valued function $P_{c}=\{P_{c}(\ell)\}_{\ell\in\Theta_{[a,b]}}\in \mathcal{H}^{n+*}_{\infty}$ defined as
$
P_{c}(\ell)=P_{i},\ \ell\in\Lambda_{i},\ i\in\overline{1, N}
$
satisfies
$\|P-P_{c}\|_{\mathcal{H}^{n\times n}_{\infty}}\leq \varepsilon_{P},$
where $\varepsilon_{P}>0$.
\end{assumption}

Use $P_{c}$ as a numerical approximation of
$P$ that fulfills \eqref{1877noLinearEP}, with the approximation error constrained by
$\|P-P_{c}\|_{\mathcal{H}^{n\times n}_{\infty}}\leq\varepsilon_{P}$. Intuitively, if $\varepsilon_{P}$ is sufficiently small, this numerical approximation would be feasible, meaning that $P_{c}$ meets \eqref{1877noLinearEP}. This assertion is formally stated below and a rigorous proof can be found in Appendix \ref{AppC}.
\begin{lemma}\label{step1}
Given scalar $\gamma>0$, suppose that Assumptions \ref{continuousAss1535}-\ref{ass1636} are fulfilled.
If the approximation error $\varepsilon_{P}$ is sufficiently small, then the function $P_{c}$ defined in Assumption \ref{ass1636} satisfies \eqref{1877noLinearEP}.
\end{lemma}

In Lemma \ref{step1}, we investigated the viability of numerically approximating the solution $P$ to the inequality \eqref{1877noLinearEP} by using the piecewise constant matrix-valued function $P_{c}$.  Beyond this,
it is also essential to explore the feasibility of the LMI-Problem stated in Definition \ref{LMIfeasP}.

Suppose that
Assumption \ref{ass1139finitegrid} is fulfilled.
Define
\begin{equation*}\label{de1867ABCQ}
\sigma_{A,B,C,Q}:=\max_{i\in\overline{1,N}}\{\sigma_{A,i}, \sigma_{B,i},\sigma_{C,i}, \sigma_{Q,i}\}.
\end{equation*}
Clearly, when the number of grid points in quantizing the state space is sufficiently large,
$\sigma_{A,B,C,Q}$ is sufficiently small.

\begin{lemma}\label{lemma1862}
Given scalar $\gamma>0$, suppose that
Assumption \ref{ass1139finitegrid} is fulfilled with $\sigma_{A,i}=\sigma_{C,i},\ i\in\overline{1, N}$
and Assumptions \ref{continuousAss1535}-\ref{ass1636} are fulfilled.
Let $\varepsilon_{P}$ be sufficiently small
such that $P_{c}$ satisfies \eqref{1877noLinearEP}.
If $\sigma_{A,B,C,Q}$ is sufficiently small, then there exist nonsingular matrices $X_{i}\in\mathbb{R}^{n\times n}$, as well as scalars $\alpha_{i}>0$, $\beta_{i}>0$, and $\rho_{i}>0, i\in\overline{1, N}$ such that LMIs \eqref{1171PLMI} are feasible with $P_{i},i\in\overline{1, N}$ defined in Assumption \ref{ass1636} .
\end{lemma}
\begin{proof}
See Appendix \ref{AppD}.
\end{proof}
\begin{remark}
If Assumption \ref{continuousAss1535} is satisfied,
it is always possible to seek a suitable
partition $H=\{h_0, h_1,\cdots, h_N\}$ and discrete points $\{\bar{h}_{i},\ i\in\overline{1, N}\}$ such that $\sigma_{A,B,C,Q}$ is sufficiently small.
\end{remark}

Now,  by linking Lemmas \ref{step1} and \ref{lemma1862},  we are ready to conclude that under certain assumptions, the gridding method proposed in Proposition \ref{promain} is feasible.
\begin{theorem}\label{SCI3the1590}
Given scalar $\gamma>0$, suppose that Assumption \ref{ass1139finitegrid} is fulfilled with $\sigma_{A,i}=\sigma_{C,i},\ i\in\overline{1, N}$ and Assumptions \ref{continuousAss1535}-\ref{ass1636} are fulfilled. If the approximation error $\varepsilon_{P}$ is sufficiently small,
then there exist a suitable
partition $H=\{h_0, h_1,\cdots, h_N\}$ and discrete points $\{\bar{h}_{i},\ i\in\overline{1, N}\}$
such that the LMI-Problem stated in Definition \ref{LMIfeasP} is feasible.
\end{theorem}
\begin{remark}
According to Proposition \ref{promain} and Theorem \ref{SCI3the1590}, it is clear that under certain assumptions, if $\varepsilon_{P}$ and $\sigma_{A,B,C,Q}$ are sufficiently small,
the conservatism of the sufficient condition in Proposition \ref{promain} will become arbitrarily small.
\end{remark}

\section{An Application to NCSs}\label{SecApptoNCS}
\begin{figure}
\centering
\vspace{-5pt}
\includegraphics[width=0.32\textwidth]{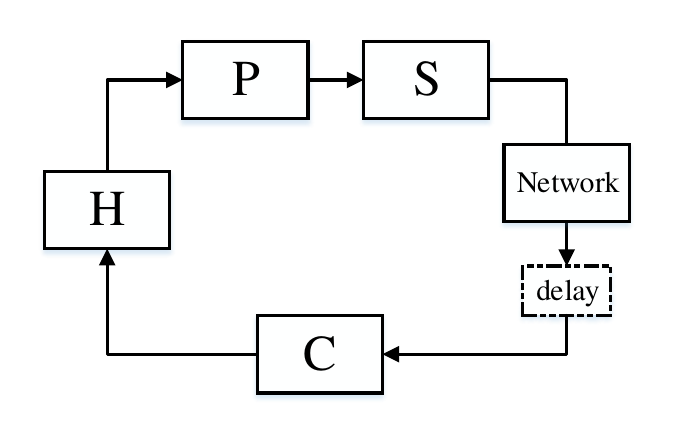}\\
\vspace{-5pt}
\caption{Networked control system.}\label{NCSs1918}
\vspace{-5pt}
\end{figure}

Consider the NCS setup shown in Fig. \ref{NCSs1918},
where the plant $P$ is a linear continuous-time system described by
\begin{equation}\label{conSys1925}
\dot{x}_{c}(t)=A_{c}x_{c}(t)+B_{c}u_{c}(t),\  x_{c}(0)=x_{0}.
\end{equation}
The sampler $S$ implements a
periodic sampling scheme with period $L$ and records the state $x_{c}$ at $t=kL,\ k\in\mathbb{N}$. The controller $C$ receives the sampled state $x_{c}(kL)$ at $t=kL+\tau(k)$, where $\tau(k)$ denotes the random network-induced delay that is assumed to smaller than the sampling period $L$ and can be known on-line.
Through the hold device $H$, the control signal in the system performs as
\begin{equation}\label{1930Control}
u_{c}(t)=
\begin{cases}
0, & t\in[0,L), \\
u(k-1), &  t\in[kL,kL+\tau(k)), \\
u(k), & t\in[kL+\tau(k),(k+1)L)
\end{cases}
\end{equation}
with
$
u(k)=
(K+\Delta)x_{c}(kL),\ k\in\mathbb{N}^{+}.
$
Here, $\Delta$ is the uncertain perturbation of the controller parameters due to physical environmental limitations,
which may include the imprecision arising from analog-to-digital and digital-to-analog conversions,
finite word length, and limited resolution measuring instruments
and roundoff errors in numerical computations (see \cite{Keel1997}).
Consider the sampled-data system composed of
\eqref{conSys1925} and \eqref{1930Control}.
One can obtain that
$x_{c}(L)=e^{A_c  L}x_{0}$
and for $k\in\mathbb{N}^{+}$,
\begin{align}
& x_{c}((k+1)L)=\int_0^{\tau(k)} e^{A_c(L-t)} B_c  u(k-1)d t \label{xc1956}\\
 &+e^{A_c L} x_{c}(kL)+\int_{0}^{L-\tau(k)} e^{A_c( L-\tau(k)-t)} B_c  u(k) d t.\nonumber
\end{align}
Let
$x_{d}(k+1)=\left[
        \begin{array}{cc}
          x_{c}(kL)^{T} &
          x_{c}((k+1)L)^{T} \\
        \end{array}
      \right]^{T}, k\in\mathbb{N}.$
Then the evolution of $x_{d}(k)$ can be described as
\begin{equation}\label{Systemxx}
\begin{split}
x_{d}&(k+1)=A_{d}(\tau(k))x_{d}(k)\\
&+B_{d}(\tau(k))\Delta (\tau(k)) C_{d}(\tau(k)) x_{d}(k),
\end{split}
\end{equation}
with the initial condition
$x_{d}(1)=\left[
        \begin{array}{cc}
          x_{0}^{T} &
          x_{c}(L)^{T} \\
        \end{array}
      \right]^{T},$
where
\begin{align*}
& A_{d}(\tau(k))=\left[
   \begin{array}{cc}
     0 & I \\
     W_{1}(\tau(k)) K &\  e^{A_c  L}+W_{2}(\tau(k)) K \\
   \end{array}
 \right],\nonumber\\
& B_{d}(\tau(k))=\left[
   \begin{array}{cc}
     0 & 0 \\
     W_{1}(\tau(k))  &W_{2}(\tau(k))  \\
   \end{array}
 \right], \nonumber\\
& \Delta (\tau(k))=
\left[
  \begin{array}{cc}
    \Delta & 0 \\
    0 & \Delta \\
  \end{array}
\right],\
C (\tau(k))=\left[
  \begin{array}{cc}
    I & 0 \\
    0 & I \\
  \end{array}
\right], \nonumber\\
&W_{1}(\tau(k))=\int_0^{\tau(k)} e^{A_c(L-t)} d t B_c,
\nonumber \\
&W_{2}(\tau(k))=\int_{0}^{L-\tau(k)} e^{A_c( L-\tau(k)-t)} d t B_c,\ k\in\mathbb{N}^{+}.
\end{align*}

We will proceed to present two examples that demonstrate the practicality of the results obtained in this paper. To simplify the model, the coefficients of plant $P$ are set to $A_{c}=0.2$ and $B_{c}=0.8,$ the state $x_{0}\in\mathbb{R}^{1}$, and the sampling period is fixed as $L=1.$ Regarding the controller parameters,
$K=-1.2$ is adopted and the uncertain perturbation is defined to satisfy $\Delta\in\mathbb{R}^{1}$.

In some real-world communication systems, the current delay $\tau(k)$ is usually determined by the previous delay $\tau(k-1)$. Therefore, modeling $\{\tau(k), k \in \mathbb{N}^{+}\}$ introduced in Fig. \ref{NCSs1918} as a Markov chain is reasonable (see \cite{ZhangLiqian2005, ShiYang2009TAC}).
In Example \ref{2073exam}, the sequence of network-induced delays $\{\tau(k), k \in \mathbb{N}^{+}\}$ is characterized as a two-mode Markov chain.

\begin{example}\label{2073exam}
In this example,
the sequence of network- induced delays $\{\tau(k), k \in \mathbb{N}^{+}\}$
is modeled as a homogeneous Markov chain that
takes values in the set $\Theta_{2}=\{0.1, 0.3\}.$ The Markov chain is determined by
the initial distribution $\pi_{i}=\mathbb{P}(\tau(1)=i)=0.5$
and the transition probability function
$$p_{ij}=\mathbb{P}(\tau(k+1)=j|\tau(k)=i)=
\begin{cases}
\frac{2}{3}, &  i=j, \\
\frac{1}{3},& i\neq j,
\end{cases}
$$
where $i,j\in\Theta_{2}$ and $k\in\mathbb{N}^{+}.$
Consider the nominal system of \eqref{Systemxx} as follows:
\begin{equation}\label{SystemNor2100}
x_{d}(k+1)=A_{d}(\tau(k))x_{d}(k),\ k\in\mathbb{N}^{+}.
\end{equation}
For any $P(i)\in \mathbb{R}^{2\times 2}$, $i\in\Theta_{2}$, and $j\in\Theta_{2}$,
denote
$\mathcal{L}_{A_{d}}(P)(j)
=\Sigma_{i\in\Theta_{2}}p_{ij}A_{d}(i)P(i)A_{d}(i)^{T}.$
According to Theorem 3.9 in \cite{BookCosta2005},
we have that \eqref{SystemNor2100} is EMSS due to that the spectral radius $r_{\sigma}(\mathcal{L}_{A_{d}})= 0.2655<1.$
Furthermore,
by Proposition \ref{finiteMarkovchain135}, it can be obtained that the uncertain system \eqref{Systemxx} is EMSS when the uncertainty $\Delta$ satisfies $\|\Delta\|\leq 0.6803.$
\end{example}

In Example \ref{2073exam}, we employed a finite state Markov chain to model the sequence of transmission delays. This restricts delays values to a countable discrete set, and the model is highly aggregated. From a practical standpoint,
the transmission delays occurring in the NCSs depicted in Fig. \ref{NCSs1918} usually take values within an interval $[\tau_{\min}, \tau_{\max}]$ (see \cite{Ioannis2014} and \cite{Masashi2018}). Consequently, modeling the sequence of delays as a continuous-valued Markov chain is more suitable for real-world applications.

\begin{example}\label{exampledelay2}
In this example, the sequence of network- induced delays  $\{\tau(k),k\in\mathbb{N}^{+}\}$ is modeled as a Markov chain that
takes values in the set $\Theta_{[0, 0.4]}=[0, 0.4].$
The Markov chain is determined by the initial distribution  $\mu_{1}(\Lambda)=\mathbb{P}(\tau(1)\in \Lambda)= 2.5\mu_{L}(\Lambda)$, $\Lambda\in\mathcal{B}(\Theta_{[0, 0.4]})$
and the stochastic kernel described by a density function
$$
g(t,s)=
\begin{cases}
\frac{5 s}{t}, & 0 \leq s<t \leq 0.4 , \\
5, &  0 \leq t=s \leq 0.4 , \\
\frac{5(2-5 s)}{2-5 t}, &  0 \leq t<s \leq 0.4 ,
\end{cases}
$$
where $\mu_{L}(\cdot)$ is the Lebesgue measure on $\mathcal{B}(\Theta_{[0, 0.4]})$.
In the following, we will apply Theorem \ref{finitegrid1166} to analyze the robust stability of the discrete-time MJLS \eqref{Systemxx}, which is equipped with the continuous-valued Markov chain $\{\tau(k),k\in\mathbb{N}^{+}\}$. Before that, we first
partition the state set of the Markov chain following the gridding method proposed in Subsection \ref{SecRobustofConCase}:
Subdivide the interval $\Theta_{[0, 0.4]}$ into 20 subintervals $\{\Lambda_{i},i\in\overline{1, 20}\}$ of  equal length $0.02,$
and designate
the midpoint of each subinterval as $\{\bar{h}_{i},\ i\in\overline{1, 20}\}.$
Through careful analysis and computation, it can be established that Assumption \ref{ass1139finitegrid} is satisfied when the scalars $\sigma_{A,i}, \sigma_{B,i}, \sigma_{C,i}$ and $\sigma_{Q,i}$ take the values indicated in the Fig. \ref{Xigmai}. Fix $\gamma = 3.1$.
By utilizing the LMI toolbox, the feasibility of LMI \eqref{1171PLMI} is then confirmed when the scalars $\alpha_{i}, \beta_{i},$ and $\rho_{i}$ are chosen as in Fig. \ref{AlphaBetaRho}, as will as the matrices $X_i$ and $P_i$ are given as in Fig. \ref{Xi} and Fig. \ref{Pi}, respectively, $i\in\overline{1, 20}.$

\begin{figure}[h!]
\centering
\vspace{-5pt}
\includegraphics[width=0.4\textwidth]{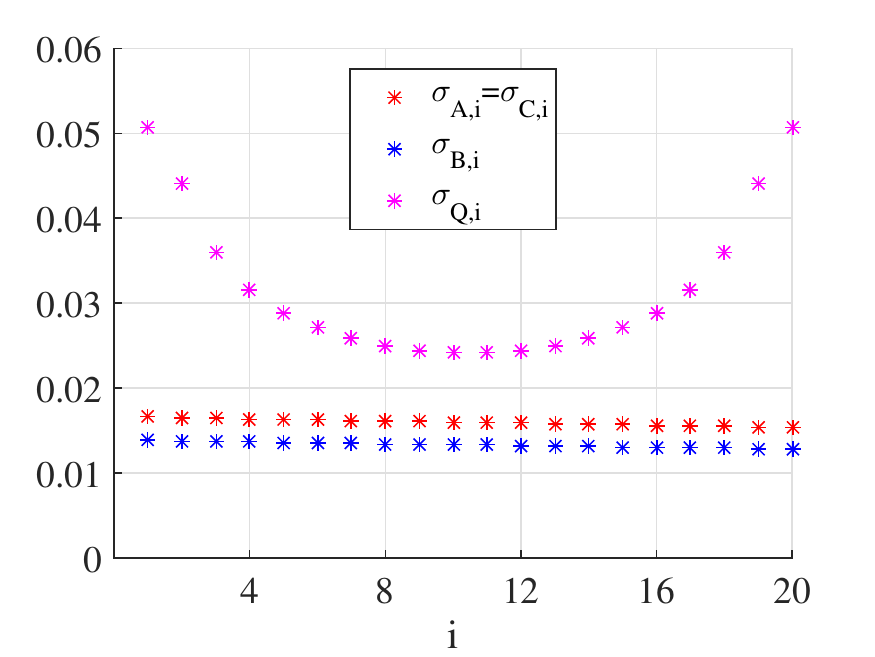}\\
\vspace{-5pt}
\caption{$\sigma_{A,i}, \sigma_{B,i}, \sigma_{C,i}$ and $\sigma_{Q,i}$.}\label{Xigmai}
\end{figure}

\begin{figure}[h!]
\centering
\vspace{-5pt}
\includegraphics[width=0.4\textwidth]{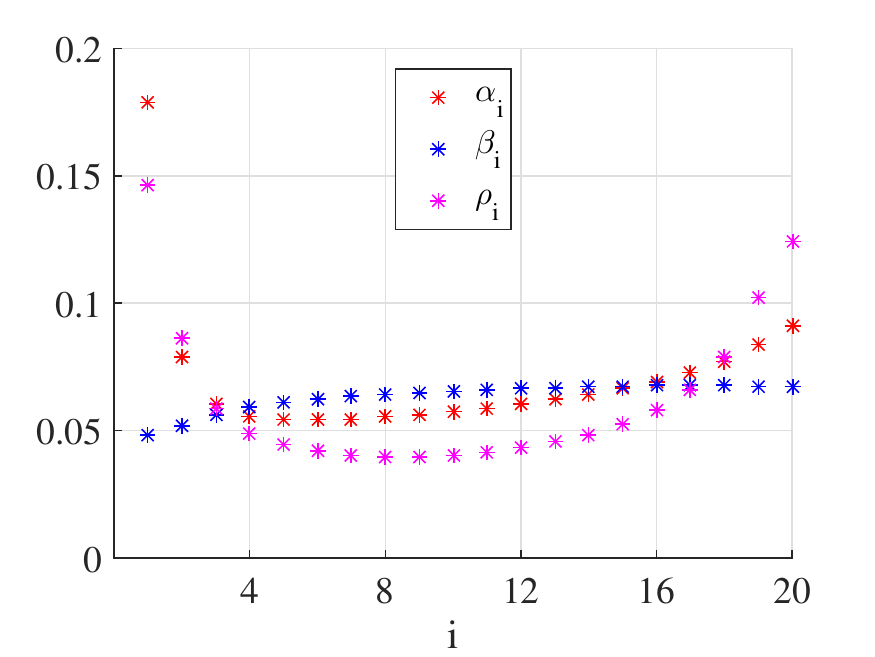}\\
\vspace{-5pt}
\caption{$\alpha_{i}, \beta_{i}$ and $\rho_{i}$.}\label{AlphaBetaRho}
\end{figure}

\begin{figure}[h!]
\centering
\vspace{-5pt}
\includegraphics[width=0.4\textwidth]{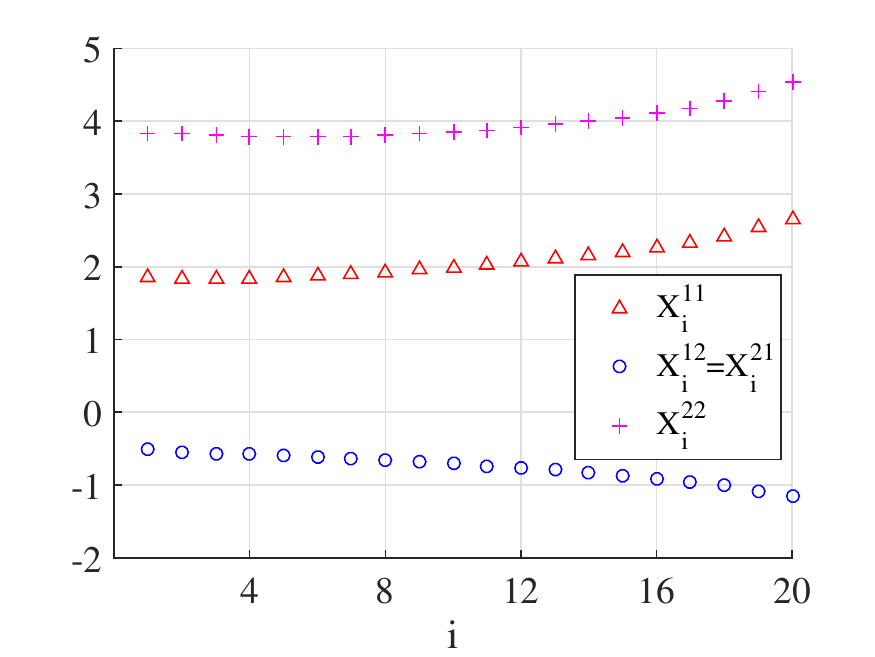}\\
\vspace{-5pt}
\caption{$X_{i}$. $X_{i}^{hg}$ is the $(h,g)$th entry of the matrix $X_{i}$, $h,g\in\{1,2\}.$}\label{Xi}
\end{figure}

\begin{figure}[h!]
\centering
\vspace{-5pt}
\includegraphics[width=0.4\textwidth]{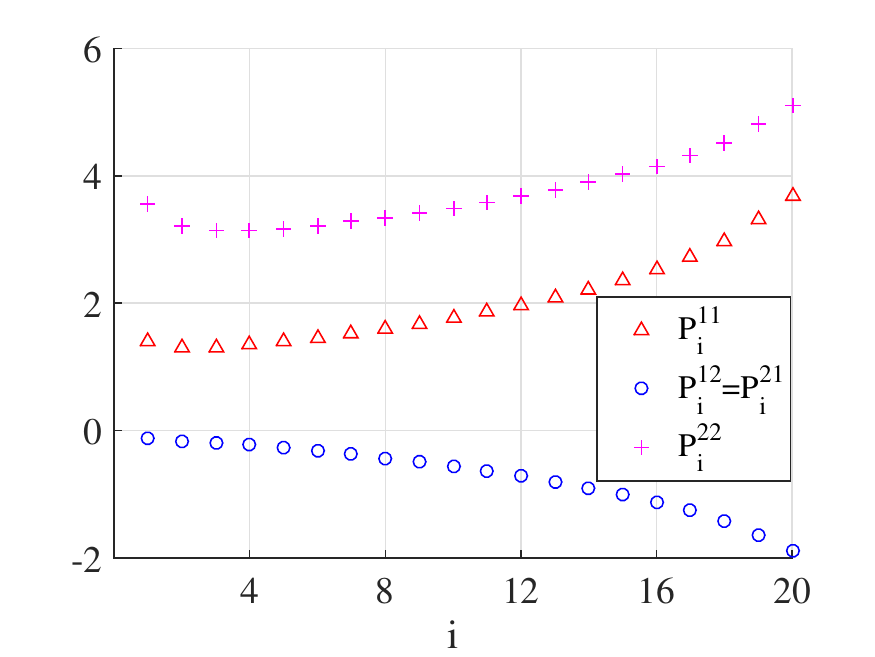}\\
\vspace{-5pt}
\caption{$P_{i}$.}\label{Pi}
\end{figure}

According to Theorem  \ref{finitegrid1166}, when the uncertainty $\Delta$ satisfies $\|\Delta\|\leq 0.568$,
the uncertain system \eqref{Systemxx} is EMSS, and it can be further known that the system \eqref{xc1956} is also EMSS.  Now let the initial value $x_{0}=-2.$  For different uncertainties $\Delta\in\{-0.568,-0.4,-0.2,0,0.2,0.4,0.568\},$ some state trajectories of the system \eqref{xc1956} governed by the control law $u(k)=(K+\Delta)x_{c}(k)$
are presented in Fig. \ref{xck}, while the mean values of the system state $x_{c}(k)$ are plotted in Fig. \ref{meanxck}.

\begin{figure}[h!]
\centering
\vspace{-5pt}
\includegraphics[width=0.4\textwidth]{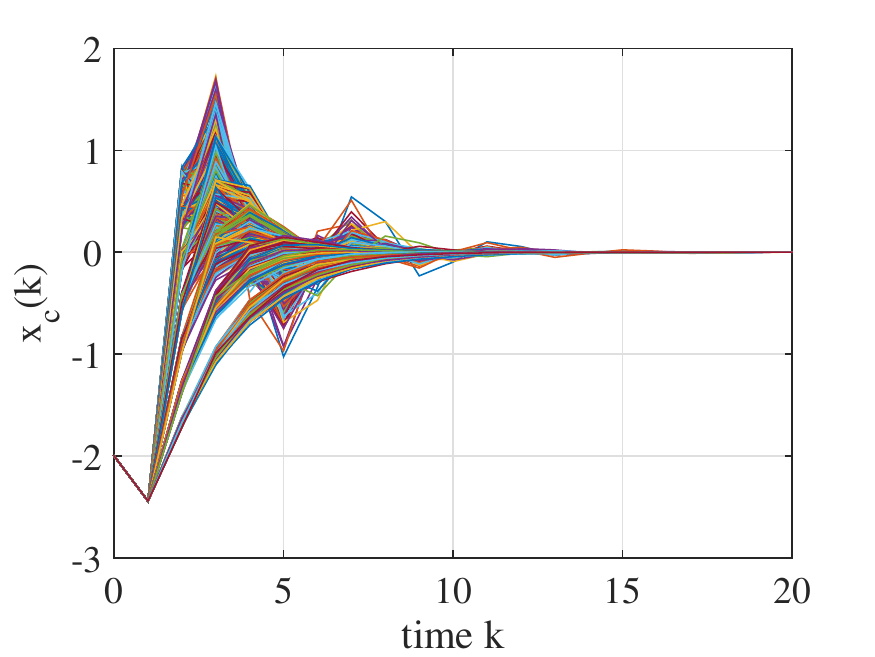}\\
\vspace{-5pt}
\caption{Some sample trajectories of the closed-loop system of \eqref{xc1956}.}\label{xck}
\end{figure}

\begin{figure}[h!]
\centering
\vspace{-5pt}
\includegraphics[width=0.4\textwidth]{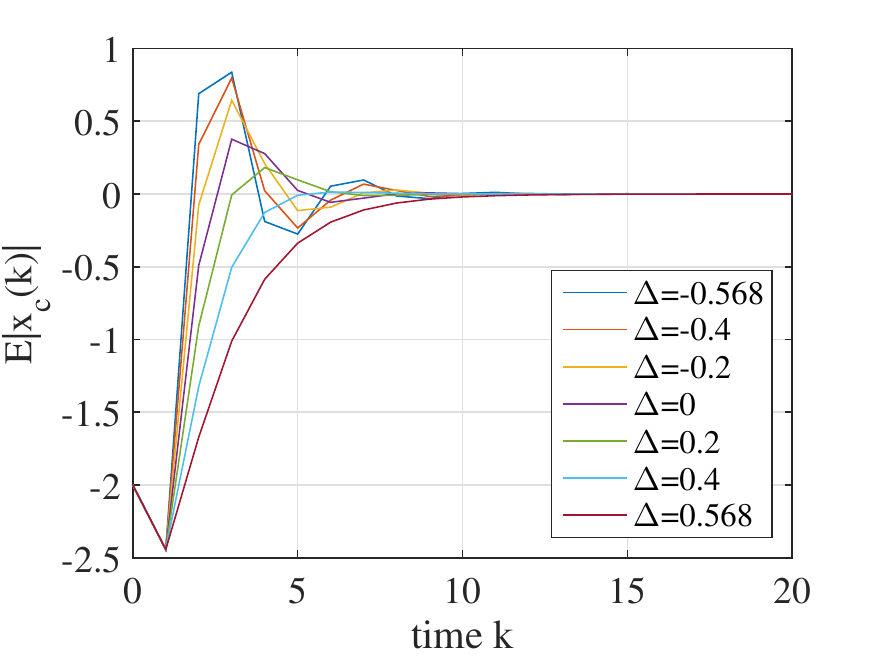}\\
\caption{The mean values of the state $x_{c}(k)$ for different uncertainties.}\label{meanxck}
\end{figure}

\end{example}

\section{Conclusions}\label{SecConclusions}
In this paper, we have analyzed the robustness of MJLSs which evolve in a discrete-time domain and are governed by the Markov chain on a Borel space.
With the help of the infinite-dimensional operator theory, an extension of the small gain theorem has been built based on the given BRL and has subsequently been exploited to give a lower bound for what we define as the stability radius. Regarding systems with continuous-valued Markov jump parameters, we have proposed a griding method and given rigorous proof of its feasibility and conservatism. This method has provided a sufficient condition for the existence of uniform positive solutions to a coupled inequality associated with $H_{\infty}$ analysis, expressed in terms of a finite number of LMIs.
This transforms the challenging infinite-dimensional problem into a finite-dimensional one, for which the resolution is obviously achievable.
Given this result, the robust stability of the systems can be determined by addressing an LMI feasibility problem.
Within the finite MJLS framework, the adjoint method discussed in \cite{Todorov2011SIAM} may provide a less conservative condition for the robust stability of continuous-time systems (\cite{Todorov2013Auto}), where the adjoint relationship is defined for operators acting on a Hilbert space. However, for discrete-time MJLSs with the Markov chain on a Borel space, the corresponding methods still require further investigation.

\begin{appendices}
\section{Proof of Proposition \ref{Schurcom}}\label{AppA}
$(iii)\Rightarrow (i)$:
Clearly, $P\in \mathcal{SH}^{(n+m)}$. Introduce $Q_{1}=\{Q_{1}(\ell)\}_{\ell\in\Theta}\in\mathcal{H}^{n\times m}_{\infty}$ defined as $Q_{1}(\ell)=P_{1}(\ell)^{-1}P_{2}(\ell).$
Noticing that for any $\ell\in\Theta$,
\begin{equation}\label{fenjie421}
\left[
       \begin{array}{cc}
         I & 0 \\
        Q_{1}(\ell)^{T}& I \\
       \end{array}
\right]
     \left[
       \begin{array}{cc}
         P_{1}(\ell) & 0 \\
        0& W_{1}(\ell) \\
       \end{array}
     \right]
\left[
       \begin{array}{cc}
         I & Q_{1}(\ell) \\
        0& I \\
       \end{array}
\right]
=P(\ell)
\end{equation}
with $W_{1}(\ell)= P_{3}(\ell)-P_{2}(\ell)^{T}P_{1}(\ell)^{-1}P_{2}(\ell),$
we arrive at
$\|P(\ell)\|_{\mathbb{R}^{(n+m)\times (n+m)}}\leq
\|P_{1}(\ell)\|_{\mathbb{R}^{n\times n}} \|W_{1}(\ell)\|_{\mathbb{R}^{m\times m}}
\leq\|P_{1}\|_{\mathcal{H}^{n\times n}_{\infty}}
\|W_{1}\|_{\mathcal{H}^{m\times m}_{\infty}}
\ \mu\text{-}a.e.,$
 which implies that $P\in \mathcal{SH}^{(n+m)}_{\infty}.$
Next, we will show that $P(\ell)\gg 0$  $\mu$-$a.e..$ It follows from $(iii)$ that  there exist $\xi_{1}>0$ and $\xi_{2}>0$ such that
$P_{1}(\ell)\geq \xi_{1} I$ $\mu$-$a.e.$ and $W_{1}(\ell)\geq \xi_{2} I \ \mu\text{-}a.e..$
Thus, by taking $\xi_{3}=\min\{\xi_{1}, \xi_{2}\}$ and bearing \eqref{fenjie421} in mind, we obtain that
$$
P(\ell)\geq \xi_{3}
\left[
       \begin{array}{cc}
         I & Q_{1}(\ell) \\
        Q_{1}(\ell)^{T}& I+Q_{1}(\ell)^{T}Q_{1}(\ell) \\
       \end{array}
\right]
\ \mu\text{-}a.e..
$$
On the other hand, observe that for any $\xi_{4}\in(0,1),$
\begin{equation*}
\begin{split}
&\left[
       \begin{array}{cc}
         I & Q_{1}(\ell) \\
        Q_{1}(\ell)^{T}& I+Q_{1}(\ell)^{T}Q_{1}(\ell) \\
       \end{array}
\right]-\xi_{4}I=\\
&\left[
       \begin{array}{cc}
         I & 0 \\
        \frac{1}{1-\xi_{4}}Q_{1}(\ell)^{T}& I \\
       \end{array}
\right]
\left[
       \begin{array}{cc}
         I-\xi_{4}I & 0 \\
         0& W_{2}(\ell) \\
       \end{array}
\right]
\left[
       \begin{array}{cc}
         I & \frac{1}{1-\xi_{4}}Q_{1}(\ell) \\
        0& I \\
       \end{array}
\right]> 0
\end{split}
\end{equation*}
$\mu\text{-}a.e.,$
where $W_{2}(\ell)= I-\xi_{4}I+(1-\frac{1}{1-\xi_{4}})Q_{1}(\ell)^{T}Q_{1}(\ell).$ Now one can choose $\xi=\xi_{3}\xi_{4}>0$ such that $P(\ell)\geq \xi I$  $\mu$-$a.e.$, and thus $(iii)\Rightarrow (i)$ is proved.\\
$(i)\Rightarrow (ii)$: From $(i)$ we know that there exists $\xi>0$ such that $P(\ell)\geq \xi I$  $\mu$-$a.e..$
Furthermore,
$$P_{3}(\ell)=\left[
                \begin{array}{cc}
                  0 & I \\
                \end{array}
              \right]
              P(\ell)
              \left[
                \begin{array}{c}
                  0 \\
                  I \\
                \end{array}
              \right]
\geq \xi \left[
                \begin{array}{cc}
                  0 & I \\
                \end{array}
              \right]
              \left[
                \begin{array}{c}
                  0 \\
                  I \\
                \end{array}
              \right]
=\xi I \ \mu\text{-}a.e.,
$$
which means $P_{3}\in \mathcal{H}^{m+*}_{\infty}$.
Now define $Q_{2}(\ell):=-P_{3}(\ell)^{-1}P_{2}(\ell)^{T},$
$W_{3}(\ell):=P_{1}(\ell)-P_{2}(\ell)P_{3}(\ell)^{-1}
P_{2}(\ell)^{T},$ $\ell\in\Theta.$  Here, it is worth emphasizing that
$P_{3}\in \mathcal{H}^{m+*}_{\infty}$ ensures  $P_{3}^{-1}\in\mathcal{H}^{m+*}_{\infty}.$
Clearly,
$Q_{2}\in\mathcal{H}^{m\times n}_{\infty}$ and
$W_{3}\in \mathcal{SH}^{n}_{\infty}$.
Note that
\begin{equation*}
\begin{split}
&\left[
    \begin{array}{cc}
      W_{3}(\ell) & 0 \\
      0 & P_{3}(\ell) \\
    \end{array}
  \right]
=\left[
   \begin{array}{cc}
     I & Q_{2}(\ell)^{T} \\
     0 & I \\
   \end{array}
 \right]
P(\ell)
\left[
   \begin{array}{cc}
     I & 0 \\
     Q_{2}(\ell) & I \\
   \end{array}
 \right]\\
 & \geq \xi
\left[
       \begin{array}{cc}
         I+Q_{2}(\ell)^{T}Q_{2}(\ell) & Q_{2}(\ell)^{T} \\
        Q_{2}(\ell)& I \\
       \end{array}
\right]\ \mu\text{-}a.e..
\end{split}
\end{equation*}
Also, for any $\xi_{5}\in(0,1),$
\begin{equation*}
\begin{split}
&\left[
       \begin{array}{cc}
         I+Q_{2}(\ell)^{T}Q_{2}(\ell) & Q_{2}(\ell)^{T} \\
        Q_{2}(\ell)& I \\
       \end{array}
\right]-\xi_{5}I
=\\
&\left[
       \begin{array}{cc}
         I & \frac{1}{1-\xi_{5}}Q_{2}(\ell)^{T} \\
        0& I \\
       \end{array}
\right]
\left[
       \begin{array}{cc}
         W_{4}(\ell) & 0 \\
         0& I-\xi_{5}I \\
       \end{array}
\right]
\left[
       \begin{array}{cc}
         I & 0 \\
        \frac{1}{1-\xi_{5}}Q_{2}(\ell) & I \\
       \end{array}
\right]> 0
\end{split}
\end{equation*}
holds $\mu$-$a.e.$,
where $W_{4}(\ell)= (1-\frac{1}{1-\xi_{5}})Q_{2}(\ell)^{T}Q_{2}(\ell)+I-\xi_{5}I.$ Therefore, there is $\xi_{6}=\xi\cdot\xi_{5}>0$ such that $W_{3}(\ell)\geq \xi_{6} I \ \mu\text{-}a.e.$, which shows the validity of the implication $(i)\Rightarrow (ii)$.\\
$(ii)\Rightarrow (iii)$: It can be deduced
from $(ii)$ that there exist $\xi_{6}>0$ and $\xi_{7}>0$ such that $W_{3}(\ell)\geq \xi_{6} I\ \mu\text{-}a.e.$ and $P_{3}(\ell)\geq \xi_{7} I\ \mu\text{-}a.e.$.
Then,
$P_{1}(\ell)
\geq P_{2}(\ell)P_{3}(\ell)^{-1}
P_{2}(\ell)^{T}+\xi_{6} I\geq \xi_{6} I\ \mu\text{-}a.e.,$ which yields that $P_{1}\in \mathcal{H}^{n+*}_{\infty},$ and further we have $W_{1}\in \mathcal{SH}^{m}_{\infty}$.
Note that
$$
\left[
       \begin{array}{cc}
         P_{1}(\ell) & 0 \\
        0& W_{1}(\ell) \\
       \end{array}
     \right]
=
U(\ell)^{T}
\left[
    \begin{array}{cc}
      W_{3}(\ell) & 0 \\
      0 & P_{3}(\ell) \\
    \end{array}
\right]U(\ell) \ \mu\text{-}a.e.,
$$
where
$
U(\ell)=\left[
    \begin{array}{cc}
      I & 0 \\
      -Q_{2}(\ell) & I \\
    \end{array}
  \right]
\left[
    \begin{array}{cc}
      I & -Q_{1}(\ell) \\
      0 & I \\
    \end{array}
  \right],
$
and it can be confirmed that there exists $\xi_{8}>0$  such that $U(\ell)^{T}U(\ell)\geq \xi_{8}I \ \mu\text{-}a.e..$
Choosing $\xi_{9}=\min\{\xi_{6},\xi_{7}\}>0$, one can conclude that there exists $\xi_{10}=\xi_{8}\xi_{9}>0$ such that
$W_{1}(\ell)\geq \xi_{10} I \ \mu\text{-}a.e.,$
i.e., $W_{1}\in \mathcal{H}^{m+*}_{\infty}$.
\end{appendices}

\begin{appendices}
\section{Proof of Proposition \ref{341proposition}}\label{Pronondegen}
Prove $(i)$  by applying induction on $k$.
First, $(i)$ is trivially satisfied for $k=0$.
It follows from \eqref{ProMa206} that for each $k\in\mathbb{N}$, $\nu_{k}\in\mathcal{H}^{1+}_{1}$.
Supposing that $\nu_{k}(\ell)>0$ $\mu\text{-}a.e.$, we assert that $\nu_{k+1}(\ell)>0$ $\mu\text{-}a.e..$
If not, there exists $\bar{\Lambda}\in{\mathcal{B}(\Theta)}$ with $\mu(\bar{\Lambda})>0$ such that $\int_{\bar{\Lambda}}\nu_{k+1}(\ell)\mu(d\ell)=0$. Using Fubini's theorem, it holds that
\begin{equation}\label{369innu}
\int_{\Theta}\nu_{k}(t)\left[\int_{\bar{\Lambda}}g(t,\ell)\mu(d\ell)\right]\mu(dt)
=\int_{\bar{\Lambda}}\nu_{k+1}(\ell)\mu(d\ell)
=0.
\end{equation}
In view of the assumption that $\nu_{k}(\ell)>0$ $\mu\text{-}a.e.$,
we can derive via \eqref{369innu} that $\int_{\bar{\Lambda}} g(t,\ell)\mu(d\ell)=0$ is true for $\mu\text{-}$almost all $t\in\Theta$.
Hence,
$$\int_{\bar{\Lambda}}\int_{\Theta}g(t,\ell)\mu(dt)\mu(d\ell)
=\int_{\Theta}\int_{\bar{\Lambda}}g(t,\ell)
\mu(d\ell)\mu(dt)=0$$
follows by Fubini's theorem, which yields
$\int_{\Theta}g(t,\ell)\mu(dt)=0$
for $\mu\text{-}$almost all $\ell\in{\bar{\Lambda}}$ with $\mu(\bar{\Lambda})>0$,
which contradicts with the fact that $\int_{\Theta}g(t,\ell)\mu(dt)>0$ $\mu\text{-}a.e.$, and $(i)$ is proved.\\
The proof of $(ii)$ will be completed by  contradiction.
Define set
$$\bar{\Lambda}:=\{\ell\in\Theta|
\int_{\Theta}g(t,\ell)\mu(dt)=0\}.$$
Clearly, $\bar{\Lambda}\in\mathcal{B}(\Theta).$
If the implication is not valid, then $\mu(\bar{\Lambda})>0,$
and since $\nu_{k+1}(\ell)>0$ $\mu\text{-}a.e.$, according to Theorem 15.2 in  \cite{BookBillingsley1995}, we have that
\begin{equation}\label{394hatLc}
\int_{\bar{\Lambda}}\nu_{k+1}(\ell)\mu(d\ell)>0.
\end{equation}
From $\int_{\Theta}\int_{\bar{\Lambda}}g(t,\ell)\mu(d\ell)\mu(dt)
=\int_{\bar{\Lambda}}\int_{\Theta}g(t,\ell)\mu(dt)\mu(d\ell)
=0,$
it follows that $\int_{\bar{\Lambda}}g(t,\ell)\mu(d\ell)=0$
$\mu\text{-}a.e..$
Hence,
$$\int_{\bar{\Lambda}}\nu_{k+1}(\ell)\mu(d\ell)=
\int_{\Theta}\nu_{k}(t)\int_{\bar{\Lambda}}g(t,\ell)\mu(d\ell)
\mu(dt)
=0.$$
This contradicts with \eqref{394hatLc}. Thus, the proof of $(ii)$ is proved.
\end{appendices}

\begin{appendices}
\section{Proof of Lemma \ref{lemmaBRL}}\label{AppB}
$(i)\Leftrightarrow (ii)$ is established in \cite{Xiao2023}. \\
$(iii)\Rightarrow (ii):$
Since $(iii)$ is valid,  then
there exists $\check{P}=-\hat{P}\in{\mathcal{H}_{\infty}^{n-*}}$ satisfying $\mathcal{W}(\check{P})(\ell)\gg 0\ \mu\text{-}a.e.$
and such that
$\check{\Psi}_{3}^{\gamma}(\check{P})(\ell)\gg 0\ \mu\text{-}a.e.,$ where
\begin{equation*}
\begin{split}
&\mathcal{W}(\check{P})(\ell)=
-\check{P}(\ell)+\mathcal{T}_{A}(\check{P})(\ell)
-C(\ell)^{T}C(\ell)\\
&\ \ \ \ \ \ \ \ \ \ \ \
-\Psi_{2}(\check{P})(\ell)
\check{\Psi}_{3}^{\gamma}(\check{P})(\ell)^{-1}
\Psi_{2}(\check{P})(\ell)^{T}
\end{split}
\end{equation*}
and
\begin{equation*}
\check{\Psi}_{3}^{\gamma}(\check{P})(\ell)=\gamma I
-D(\ell)^{T}D(\ell)+\mathcal{T}_{B}(\check{P})(\ell),\ \ell\in\Theta.
\end{equation*}
Certainly, $\gamma I
-D(\ell)^{T}D(\ell)\gg 0 \ \mu\text{-}a.e..$
Now applying Proposition \ref{Schurcom}, we have that
there exists $\check{P}\in{\mathcal{H}_{\infty}^{n-*}}$
such that
\begin{equation*}
\left[
  \begin{array}{cc}
    -\check{P}(\ell)+\mathcal{T}_{A}(\check{P})(\ell)
-C(\ell)^{T}C(\ell) & \Psi_{2}(\check{P})(\ell) \\
    \Psi_{2}(\check{P})(\ell)^{T} & \check{\Psi}_{3}^{\gamma}(\check{P})(\ell) \\
  \end{array}
\right]\gg 0\ \mu\text{-}a.e..
\end{equation*}
Then it follows from  Corollary 2 in \cite{Xiao2024} that \eqref{lmiARE}
admits a unique stabilizing solution $\bar{P}$ satisfying $\Psi_{3}^{\gamma}(\bar{P})(\ell)\gg0\ \mu\text{-}a.e.$. \\
The proof of the implication $(i)\Rightarrow (iii)$ and  the relationship
$\hat{P}-\bar{P}\in\mathcal{H}_{\infty}^{n+*}$
can be referenced in the analogous arguments presented in \cite{BookDragan2010} for finite MJLSs, as well as in \cite{Souza1992SCL} for linear time-invariant systems.
\end{appendices}

\begin{appendices}
\section{Proof of Proposition \ref{usedinSMG}}\label{AppenPro2}
Analogously to the proof of Lemma 3 in \cite{Xiao2023}, it can be proved that when $\|\mathfrak{L}\|_{\infty}<1$, we have that $I-D(\ell)^{T}D(\ell)\geq \eta_{1} I\ \mu\text{-}a.e.$ for some $ 0<\eta_{1}<1.$ This implies that $\|D\|_{\mathcal{H}^{r\times r}_{\infty}}\leq \sqrt{1-\eta_1}$. Hence, from Theorem A.4.9 in \cite{Curtain1995}, we know that
$(\mathcal{I}-D)^{-1}$ exists and $$\|(\mathcal{I}-D)^{-1}\|_{\mathcal{H}^{r\times r}_{\infty}}\leq (1-\|D\|_{\mathcal{H}^{r\times r}_{\infty}})^{-1}<+\infty,$$
where $\mathcal{I}=\{\mathcal{I}(\ell)\}_{\ell\in\Theta}
\in\mathcal{SH}_{\infty}^{n}$ with $\mathcal{I}(\ell)=I$.
Moreover, it is obtained that $(\mathfrak{I}-\mathfrak{L})$ is invertible due to the hypothesis of $\|\mathfrak{L}\|_{\infty}<1$.
Based on Theorem 3 in \cite{Xiao2024},
the desired result can be established  by following similar arguments as in the proof of Theorem 4.1 in \cite{Aberkane2015SIAM}.

\end{appendices}

\begin{appendices}
\section{Proof of Lemma \ref{step1}}\label{AppC}
For any $P\in\mathcal{SH}_{\infty}^{n+}$, define the operator
\begin{equation*}
\begin{split}
W(P)(\ell)=
&P(\ell)-\Psi_{1}(P)(\ell)-\Psi_{2}(P)(\ell)
[\gamma I-\mathcal{T}_{B}(P)(\ell)]^{-1}\\
&\cdot
\Psi_{2}(P)(\ell)^{T},\ \ell\in\Theta_{[a,b]}.
\end{split}
\end{equation*}
Since $P\in\mathcal{H}_{\infty}^{n+*}$ is a solution to the inequality \eqref{1877noLinearEP}, according to Propositions \ref{Schurcom} and  \ref{pro1091},
there exist scalars $\xi_{0}>0$, $\xi_{1}>0$, and $\xi_{2}>0$ such that
$P(\ell)\geq \xi_{0} I,$
$\gamma I-\mathcal{T}_{B}(P)(\ell)\geq \xi_{1} I,$
and
$W(P)(\ell) \geq \xi_{2} I\ \mu\text{-}a.e..$
Pick $\varepsilon_{P}\in(0,\xi_{0}),$ and from $\|P-P_{c}\|_{\mathcal{H}^{n\times n}_{\infty}}\leq \varepsilon_{P}$, it follows that
$P_{c}\in \mathcal{H}^{n+*}_{\infty}.$
Note that $\|\mathcal{T}_{B}(P_{c}-P)(\ell)\|_{\mathbb{R}^{r\times r}}\leq  \varepsilon_{P}\|B\|_{\mathcal{H}_{\infty}^{n\times r}}^{2},$
and further let $\varepsilon_{P}$ satisfy $\xi_{3}>0$, where
$
\xi_{3}=\xi_{1}-\varepsilon_{P}\|B\|_{\mathcal{H}_{\infty}^{n\times r}}^{2}.
$
We get
\begin{equation}\label{1493signcon}
\gamma I-\mathcal{T}_{B}(P_{c})(\ell)
\geq \xi_{3}I \ \mu\text{-}a.e..
\end{equation}
Then, performing algebraic manipulations, we arrive at
\begin{equation}\label{eq1512}
\begin{split}
&W(P_{c})(\ell)-W(P)(\ell)\\
&=-\{[P(\ell)-P_{c}(\ell)]
-[A(\ell)-B(\ell)\bar{\mathcal{F}}^{\gamma}(P)(\ell)]^{T}\\
&\ \ \ \cdot\mathcal{E}(P-P_{c})(\ell)
[A(\ell)-B(\ell)\bar{\mathcal{F}}^{\gamma}(P)(\ell)]\}\\
&\ \ \ -[\bar{\mathcal{F}}^{\gamma}(P_{c})(\ell)
-\bar{\mathcal{F}}^{\gamma}(P)(\ell)]^{T}
[\gamma I-\mathcal{T}_{B}(P_{c})(\ell)]\\
&\ \ \ \cdot
[\bar{\mathcal{F}}^{\gamma}(P_{c})(\ell)
-\bar{\mathcal{F}}^{\gamma}(P)(\ell)]\ \mu\text{-}a.e..
\end{split}
\end{equation}
Define
$\xi_{4}:=\xi_{2}-\varepsilon_{P}(1+
\|A-B\bar{\mathcal{F}}^{\gamma}(P)\|^{2}
_{\mathcal{H}_{\infty}^{n\times n}})
-\xi_{3}^{-2}\gamma\varepsilon_{P}^{2}
\|A\|_{\mathcal{H}_{\infty}^{n\times n}}^{2}
\|B\|_{\mathcal{H}_{\infty}^{n\times r}}^{2}
(1+\xi_{1}^{-1}\|B\|_{\mathcal{H}_{\infty}^{n\times r}}^{2}\|P\|_{\mathcal{H}_{\infty}^{n\times n}})^{2}.$
By recalling \eqref{eq1512} and $W(P)(\ell)\geq \xi_{2} I $, it holds that
\begin{equation}\label{1574ieq}
W(P_{c})(\ell)
\geq \xi_{4}I\ \mu\text{-}a.e..
\end{equation}
Now select $\varepsilon_{P}$ sufficiently small such that $0<\varepsilon_{P}<\xi_{0}$ together with the conditions $\xi_{3}>0$ and $\xi_{4}>0$ hold.
Combining \eqref{1493signcon} with
\eqref{1574ieq}, then using Proposition \ref{Schurcom} twice, we conclude that \eqref{1877noLinearEP} is valid with the given $P_{c}$.
\end{appendices}

\begin{appendices}
\section{Proof of Lemma \ref{lemma1862}}\label{AppD}
Let $X_{i}=\mathcal{E}(P_{c})(\bar{h}_{i})$,  $i\in\overline{1, N}$, where $\{\bar{h}_{i},\ i\in\overline{1, N}\}$
is the sequence of selected discrete points.
Since $P_{c}$ satisfies inequality \eqref{1877noLinearEP},
it can be checked that there exists  $\xi_{5}>0$ such that
\begin{equation}\label{1698LMI8r}
\Xi^{8}_{i}(\ell)+\diag\{ [X_{i}\mathcal{E}(P_{c})(\ell)^{-1}X_{i}-X_{i}],0,0,0\}
\geq \xi_{5}I
\end{equation}
holds for $\mu$-almost all $\ell\in\Lambda_{i}$ and $i\in\overline{1, N}$, where
\begin{equation*}
\begin{split}
\Xi^{8}_{i}(\ell)=\left[
  \begin{array}{cccc}
    X_{i} & 0 & X_{i}A(\ell) & X_{i}B(\ell)\\
    0 & I &  C(\ell) & 0 \\
   A(\ell)^{T}X_{i} & C(\ell)^{T} & P_{i} & 0 \\
   B(\ell)^{T}X_{i} & 0 & 0 & \gamma I \\
  \end{array}
\right].
\end{split}
\end{equation*}
Bearing $P\in\mathcal{H}_{\infty}^{n+*}$ in mind, we get that there
exists $\xi_{0}>0$ such that
$\xi_{0} I \leq P(\ell)\leq \|P\|_{\mathcal{H}_{\infty}^{n\times n}} I\ \mu\text{-}a.e..$
Then, given that $\|P-P_{c}\|_{\mathcal{H}^{n\times n}_{\infty}}\leq \varepsilon_{P}$, where $\varepsilon_{P}\in(0,\xi_{0})$,
we obtain
$$
0<\xi_{0}-\varepsilon_{P}\leq \|\mathcal{E}(P_{c})(\ell)\|_{\mathbb{R}^{n\times n}}\leq \|P\|_{\mathcal{H}_{\infty}^{n\times n}}+\varepsilon_{P} \ \mu\text{-}a.e..
$$
Moreover, from the fact that
$$
\mathcal{E}(P_{c})(\ell)-X_{i}
=[Q(\ell)-Q_{i}]\Upsilon[Q(\ell)+Q_{i}]^{T}
$$
holds for $\mu$-almost all $\ell\in\Lambda_{i}$ and $i\in\overline{1, N}$,
it follows that
$$
\|\mathcal{E}(P_{c})(\ell)-X_{i}\|_{\mathbb{R}^{n\times n}}
\leq2\sqrt{N}\sigma_{A,B,C,Q}(\|P\|_{\mathcal{H}_{\infty}^{n\times n}}+\varepsilon_{P})
$$
holds for $\mu$-almost all $\ell\in\Lambda_{i}$, $i\in\overline{1, N}$.
Thus,
\begin{equation*}
\begin{split}
&\|X_{i}\mathcal{E}(P_{c})(\ell)^{-1}X_{i}-X_{i}\|_{\mathbb{R}^{n\times n}}\\
&\leq \|X_{i}\|_{\mathbb{R}^{n\times n}}^{2}\|\mathcal{E}(P_{c})(\ell)^{-1}-X_{i}^{-1}\|_{\mathbb{R}^{n\times n}}\\
&\leq2\sqrt{N}\sigma_{A,B,C,Q}(\xi_{0}-\varepsilon_{P})^{-2} (\|P\|_{\mathcal{H}_{\infty}^{n\times n}}+\varepsilon_{P})^{3}
\end{split}
\end{equation*}
holds for $\mu$-almost all $\ell\in\Lambda_{i}, i\in\overline{1, N}$.
Define
$$\xi_{6}=\xi_{5}- 2\sqrt{N}\sigma_{A,B,C,Q}(\xi_{0}-\varepsilon_{P})^{-2} (\|P\|_{\mathcal{H}_{\infty}^{n\times n}}+\varepsilon_{P})^{3}$$
 and
let $\sigma_{A,B,C,Q}$ satisfy $\xi_{6}>0.$
Now \eqref{1698LMI8r} allows us to derive that
$
\Xi^{8}_{i}(\ell)\geq \xi_{6}I
$
holds for $\mu$-almost all $\ell\in \Lambda_{i}$, $i\in\overline{1, N}.$
On the basis of Proposition \ref{Schurcom},
this is equivalent to that there exists $\xi_{7}>0$ such that
\begin{equation}\label{1699LMI}
\begin{split}
\Pi^{2}_{i}
\geq
\left[
  \begin{array}{cc}
    \xi_{7}I & \Xi^{9}_{i}(\ell) \\
    \Xi^{9}_{i}(\ell)^{T} & \xi_{7}I \\
  \end{array}
\right]
\end{split}
\end{equation}
holds for $\mu$-almost all $\ell\in \Lambda_{i}$, $i\in\overline{1, N}$, where $\Pi^{2}_{i}$ is given by \eqref{1552pi2} with $X_{i}^{T}=X_{i}$ and
\begin{equation*}
\begin{split}
&\Xi^{9}_{i}(\ell)=\\
&\left[
                     \begin{array}{ccc}
                       -\sigma_{A,i}X_{i}\Gamma_{A,i}(\ell) & -\sigma_{B,i}X_{i}\Gamma_{B,i}(\ell) & -\sigma_{Q,i}\Gamma_{Q,i}(\ell)\Upsilon \\
                       -\sigma_{C,i}\Gamma_{C,i}(\ell) & 0 & 0 \\
                     \end{array}
                   \right].
\end{split}
\end{equation*}
Note that
$\Xi^{9}_{i}(\ell)
\Xi^{9}_{i}(\ell)^{T}
\leq\sigma_{A,B,C,Q}^{2}
\Xi^{10}_{i}(\ell)\Xi^{10}_{i}(\ell)^{T},
$
where
$
\Xi^{10}_{i}(\ell)=\left[
   \begin{array}{ccc}
     X_{i}\Gamma_{A,i}(\ell) & X_{i}\Gamma_{B,i}(\ell) & \Gamma_{Q,i}(\ell)\Upsilon \\
     \Gamma_{C,i}(\ell) & 0 & 0 \\
   \end{array}
 \right].
$
Choosing $\sigma_{A,B,C,Q}$ sufficiently small such that there exists $\xi_{8}>0$ satisfying
$\xi_{7}I
-\xi_{7}^{-1}\sigma_{A,B,C,Q}^{2}
\Xi^{10}_{i}(\ell)\Xi^{10}_{i}(\ell)^{T}\geq \xi_{8} I,$
we get that
$$\xi_{7}I
-\xi_{7}^{-1}\Xi^{9}_{i}(\ell)
\Xi^{9}_{i}(\ell)^{T}
\geq\xi_{8} I.$$
By applying Proposition \ref{Schurcom} and then considering \eqref{1699LMI}, it can be derived that there exists $\xi_{9}>0$ such that
$
\Pi^{2}_{i}
\geq\xi_{9} I,\ i\in\overline{1, N}.
$
Clearly, when $\sigma_{A,B,C,Q}$ is sufficiently small, there exist scalars $\alpha_{i}>0, \beta_{i}>0,$ and  $\rho_{i}>0$ such that $\Pi^{3}_{i}<\frac{1}{2}\xi_{9}I$, $i \in \overline{1, N}$.
According to Proposition \ref{1552pro}, when  $X_{i}=X_{i}^{T},\ i\in\overline{1, N}$, the LMIs \eqref{1171PLMI}  are equivalent to \eqref{1534LMIs}.
Therefore, we conclude that  the LMIs \eqref{1171PLMI}  are feasible.
\end{appendices}


\end{document}